\documentclass[reqno]{amsart}
\usepackage{amsmath,amssymb}
\usepackage{graphics,epsfig,color}
\usepackage[mathscr]{euscript}

\newtheorem{theorem}{Theorem}[section]
\newtheorem{proposition}[theorem]{Proposition}
\newtheorem{lemma}[theorem]{Lemma}
\newtheorem{corollary}[theorem]{Corollary}

\numberwithin{equation}{section}
\numberwithin{theorem}{section}

\renewcommand{\epsilon}{\varepsilon}

\newcommand{\mc}[1]{{\mathcal #1}}
\newcommand{\bb}[1]{{\mathbb #1}}

\newcommand{\bs}[1]{{\boldsymbol #1}}
\newcommand{\eps}{\varepsilon}

\newcommand{\lp}{\langle}
\newcommand{\rp}{\rangle}

\newcommand{\id}{{\bs 1}}

\newcommand{\rmd}{\mathrm{d}}
\newcommand{\rme}{\mathrm{e}}

\title[Asymptotics of the $\phi^4_1$ measure in the sharp interface limit]
{Asymptotics of the $\phi^4_1$ measure \\ in the sharp interface limit}
\author[L.\ Bertini]{Lorenzo Bertini}
\address{Lorenzo Bertini \hfill\break \indent
Dipartimento di Matematica, Universit\`a di Roma La Sapienza
\hfill\break \indent
P.le Aldo Moro 5, 00185 Roma, Italy}
\email{bertini@mat.uniroma1.it}
\author[P.\ Butt\`a]{Paolo Butt\`a}
\address{Paolo Butt\`a \hfill\break \indent
Dipartimento di Matematica, Universit\`a di Roma La Sapienza,
\hfill\break \indent
P.le Aldo Moro 5, 00185 Roma, Italy}
\email{butta@mat.uniroma1.it}
\author[G.\ Di Ges\`u]{Giacomo Di Ges\`u}
\address{Giacomo Di Ges\`u\hfill\break \indent
Dipartimento di Matematica, Universit\`a di Roma La Sapienza,
\hfill\break \indent
P.le Aldo Moro 5, 00185 Roma, Italy}
\email{giacomo.digesu@uniroma1.it}

\begin{document}

\begin{abstract}
We consider the $\phi^4_1$ measure in an interval of length $\ell$, defined by a symmetric double-well potential $W$ and inverse temperature $\beta$. Our results concern its asymptotic behavior in the joint limit $\beta, \ell \to \infty$, both in the subcritical regime $\ell \ll \rme^{\beta C_W}$ and in the supercritical regime $\ell \gg \rme^{\beta C_W}$, where $C_W$ denotes the surface tension. In the former case, in which the measure concentrates on the pure phases, we prove the corresponding large deviation principle. The associated rate function is the Modica-Mortola functional modified to take into account the entropy of the locations of the interfaces. Further, we provide the sharp asymptotics of the probability of having a given number of transitions between the two pure phases. In the supercritical regime, the measure does not longer concentrate and we show that the interfaces are asymptotically distributed according to a Poisson point process.
\end{abstract}

\maketitle
\thispagestyle{empty}

\section{Introduction}
\label{sec:1}

The van der Waals' theory of phase transition~\cite{vdw} is based on the
functional
\[
\mc F(\phi) = \int\! \rmd x \, \Big[ \frac 12  |\nabla \phi(x)|^2 +  W(\phi(x)) \Big],
\]
where the scalar field $\phi$ represents the local order parameter and $W$ is a smooth, symmetric, double well potential whose minimum value, chosen to be zero, is attained at $\pm 1$, interpreted as the pure phases of the system. We assume $W''(\pm 1)>0$. After the celebrated Modica-Mortola result~\cite{MM}, the so-called gradient theory of phase transitions essentially amounts to the analysis of the sharp interface limit of the free energy functional $\mc F$, see~\cite{Alberti} for a review. In the sequel, we restrict the discussion to the case $x\in [0,\ell ]$ with free boundary conditions at the endpoints and denote by $\mc F_\ell$ the corresponding free energy functional.

The above variational formulation of phase transitions does not take into account the microscopic fluctuations, which are relevant in  various phenomena. At the mesoscopic level, the effect of fluctuations can be encoded by considering the probability measure, on the space of order parameter profiles, informally given by
\begin{equation}
\label{muinf}
\rmd \mu_{\beta,\ell}  \propto \mc D \phi \, \exp\{- \beta \mc F_\ell(\phi)\}.
\end{equation}
Inspired by the paradigmatic case $W(\phi) = \frac 12(\phi^2-1)^2$, we refer to $\mu_{\beta,\ell}$ as the $\phi^4_1$ measure, where the subscript one stands for the space dimension and the superscript refers to the growth of $W$.  In general, the probability $\mu_{\beta,\ell}$ corresponds to the Euclidean version of the quantum anharmonic oscillator and it has been extensively analyzed since it exhibits an interesting behavior in a simple setting, see \cite{SimonFI} and references therein. 

The present purpose is to analyze the behavior of the probability $\mu_{\beta, \ell}$ in the joint limit $\beta, \ell \to \infty$. The relevant length scale is $\bar \ell \approx \exp\{\beta C_W\}$, where $C_W$ is the surface tension in the Van der Waals theory. Indeed, for $\ell \ll \bar \ell$ the typical behavior of the system is described by the pure phases $\phi_\pm = \pm 1$. In particular, due to the symmetry of the boundary conditions, in this regime $\mu_{\beta, \ell}$ converges to the convex combination of Dirac measures $\frac 12 \delta_{\phi_-} + \frac 12 \delta_{\phi_+}$. A relevant question is then the analysis of fluctuations. When $\ell$ grows sub-exponentially in $\beta$ these are captured by the Modica-Mortola functional. On the other hand, as discussed in~\cite{Oea}, when $\ell$ grows exponentially, the additional entropy induced by the large system size in~\eqref{muinf} cannot be neglected, so that the probability of a profile with $n$ interfaces is of the order $\ell^n \rme^{-n \beta C_W}$. We here complete this picture by showing that the probability $\mu_{\beta, \ell}$, when suitably rescaled, satisfies a large deviation principle with a modified rate functional that takes into account the additional entropic contribution. Within the regime $\ell \ll \bar \ell$, we also refine the result of~\cite{Oea} by obtaining the sharp asymptotics of the $\mu_{\beta, \ell}$-probability  of suitable neighborhoods of profiles with $n$ interfaces. As a preliminary step, we derive the sharp behavior of the length scale $\bar \ell$, that we characterize in terms of the spectral gap of the quantum anharmonic oscillator in the semiclassical limit. In particular, we show that $\bar \ell \propto \beta^{-1/2}  \exp\{\beta C_W\}$.

In contrast, in the regime $\ell \gg \bar \ell$, there is enough entropy so that the  $\mu_{\beta, \ell}$-probability of having at least one interface does not vanish. Accordingly, we show that the limiting location of the interfaces is distributed according to a Poisson point process. We refer to~\cite{COP, SchonTan} for the analogous statement in the context of one-dimensional lattice models.

We remark that the sharp asymptotic analysis of the probability measure $\mu_{\beta, \ell}$ presented here is a preliminary step for getting sharp estimates on the metastable behavior, in the joint limit $\beta, \ell \to \infty$, of the stochastic Allen-Cahn equation, whose unique stationary distribution if $\mu_{\beta, \ell}$. While for $\ell$ fixed and $\beta \to \infty$ the corresponding Eyring-Kramers formula has been proved in~\cite{BergGentz,BDGW,BrooksDG}, the case of diverging volume presents the typical features of sharp interface limits. We refer to~\cite{ReE} for a heuristic discussion and to~\cite{KimSeo} for the analogous problem in the context of lattice models.

\section{Notation and results}
\label{sec:2}

Given strictly positive sequences $a_n,b_n$, $n\in \bb N$, we use the standard asymptotic notation $a_n\sim b_n$ when  $a_n = b_n (1 + r_n)$, with $ \lim_{n}r_n =0$. Given two sequences $a_n(x)$, $b_n(x)$ depending on the parameter $x\in \mc X$, we say $a_n(\cdot) \sim b_n(\cdot)$ uniformly in $\mc X$ when $a_n (x) = b_n (x) [1 + r_n(x)]$, with $ \lim_{n} \sup_{x\in \mc X} |r_n(x) |=0$. For $p\in [1, \infty)$, $\delta>0$, and $\mc A \subset L^p((0,1))$, the $\delta$-neighborhood of $\mc A$ is denoted by  $\mc O^p_\delta(\mc A)$.

For the love of concreteness, we here stick to the case of the paradigmatic quartic choice of the double well potential, namely,
\begin{equation}
\label{W=}
W(\phi) := \frac 12 (\phi^2-1)^2.
\end{equation}
For the sake of readability, we however write the dependence on $W$ of the relevant constants in the general case of symmetric double well potential.
As the arguments in the proofs are based on the analysis of stochastic processes, we denote by $t$ the space variable in $[0, \ell]$ and by $s$ the rescaled variable in $[0,1]$. Correspondingly, we denote by $X$ the order parameter.

For $\ell>0$, we define the map $\imath_\ell \colon (0, \ell) \to (0,1) $ as the dilation $t\mapsto t/\ell$ and use the same notation for its lift to functions, i.e., $(\imath_\ell X )(s) = X(\ell s)$, $s\in (0,1)$. We then define the rescaled
free energy functional
\begin{equation}
\label{1.3}
F_\ell (X) := \mc F_{\ell} \big(\imath_\ell^{-1} (X) \big) = \int_0^1\!\rmd s \, \left[ \frac{1}{2\ell} \dot X_s^2  + \ell\, W(X_s) \right],
\end{equation}
that we regard as a lower semicontinuous functional on $L^p((0,1))$, $p\in[1,\infty)$, understanding that $F_\ell(X) = + \infty$ if $X$ does not belong to $H^1((0,1))$. By the celebrated Modica-Mortola result \cite{MM}, in the limit $\ell\to \infty$ the $\Gamma$-limit of the sequence $(F_\ell)_{\ell>0}$ is the functional $F\colon L^p((0,1))\to [0,+\infty]$ defined by
\begin{equation}
\label{F0}
F(X) = \begin{cases}  C_\mathrm{W} |S(X)| &  \textrm{if } X\in BV( (0,1);\{-1,1\})  \\ + \infty  & \textrm{otherwise}, \end{cases} 
\end{equation}
where, for $X\in BV( (0,1))$, the space of real-valued functions of bounded variation on $(0,1)$, we have denoted by $S(X)$ the \emph{jump set} of $X$ and by $|S(X)|$ its cardinality, see, e.g.,~\cite{Alberti}. Finally, $C_\mathrm{W}$ is the surface tension in the van der Waals theory, given by
\begin{equation}
\label{tensup}
C_\mathrm{W}:= \int_{-1}^1\!\rmd x\, \sqrt{2W(x)}.
\end{equation}
In particular, $C_\mathrm{W} =  4/3$ for the choice~\eqref{W=} of the double well potential.

We next define precisely the $\phi^4_1$ measure on the interval $(0,\ell)$ with inverse temperature $\beta$ and free boundary conditions, informally introduced in~\eqref{muinf}. Let $D^2_\mathrm{N}$ be the realization of the second derivative in $L^2((0,\ell))$ with Neumann boundary condition. For $\beta\in(0,\infty)$ consider the trace class operator $C_{\beta,\ell}= \beta^{-1} (- D^2_\mathrm{N} + 1)^{-1}$ and denote by $\mu^0_{\beta,\ell}$ the Gaussian measure on $C\big([0,\ell]\big)$  with mean zero and covariance $C_{\beta,\ell}$.

Setting $W_0(x) := \frac 12 x^2$ and $\hat W := W -W_0$, the $\phi^4_1$
measure $\mu_{\beta,\ell}$ is then defined in terms of the Radon-Nykodym derivative with respect to $\mu^0_{\beta,\ell}$ by
\begin{equation}
\label{mud}
\rmd \mu_{\beta,\ell} = \frac 1{Z_{\beta,\ell}} \exp\Big\{ - \beta \int_{0}^\ell\! \rmd t\, \hat W(\cdot)  \Big\}\, \rmd\mu_{\beta,\ell}^0, \qquad Z_{\beta,\ell} =\mu_{\beta,\ell}^0 \Big( \rme^{- \beta\int_{0}^\ell \rmd t\, \hat W(\cdot)}   \Big).
\end{equation}
The present purpose is to investigate the behavior of the probability $\mu_{\beta,\ell}$ in the joint limit $\beta,\ell\to\infty$.  Let $\bar\ell$ be the typical length of the transition between the pure phases. As we show in Theorem~\ref{t:sht2}  below, it is given by
\begin{equation}
\label{barl}
\bar\ell= \bar\ell_\beta := \frac 2{A_\mathrm{W} \sqrt{\beta}} \exp\{\beta C_\mathrm{W}\},
\end{equation}
where
\begin{equation}
\label{AW}
A_\mathrm{W}:=  \frac{2\sqrt 2}{\sqrt \pi} \sqrt{W(0) \sqrt{W''(1)} } \exp\bigg\{  -  \int_0^1\! \rmd x\, \frac{W'(x) - \sqrt{2 W''(1)W(x)}}{ 2W(x)}\bigg\} .
\end{equation}
Since $W$ has a quadratic minimum for $x=1$, the function inside the integral in~\eqref{AW} is regular. For the standard choice~\eqref{W=}, $A_\mathrm{W} =  8 \sqrt{2/\pi}$. We refer, e.g., to \cite[Eq.~(4.5.22)]{HelfferLNM} for an equivalent definition of the constant $A_W$.

As we next discuss, the asymptotic behavior of $\mu_{\beta,\ell}$ is quite different in the regimes $\ell\ll\bar\ell$ and $\ell\gg\bar\ell$.

\subsection*{Large deviations in the regime $\boldsymbol{1\ll \ell \ll\bar\ell}$}

In this situation, by the symmetry of the $\phi^4_1$ measure with free boundary conditions with respect to $X\mapsto -X$, the probability $\mu_{\beta,\ell}$ concentrates on the pure phases $\pm 1$. Our first result establishes the corresponding large deviation principle. To this end, fix a sequence $(\ell_\beta)_{\beta>0}$ satisfying
\begin{equation}
\label{ipell}
\lim_{\beta\to \infty}  \beta^{-1}\ell_\beta =\infty \quad \textrm{ and } \quad \lim_{\beta\to\infty} \beta^{-1}   \log \ell_\beta = \alpha \in [0,C_\mathrm{W}).
\end{equation}
The first condition above is assumed for technical reasons and in fact the large deviation statement should hold whenever $\ell_\beta \to \infty$.

Recalling that the map $\imath_\ell$ is defined above~\eqref{1.3}, we shorthand $\imath_\beta=\imath_{\ell_\beta}$ and introduce the rescaled $\phi^4_1$ measure as
\begin{equation}
\label{1.10}
\mu_\beta := \mu_{\beta,\ell_\beta} \circ \imath_{\beta}^{-1},
\end{equation}
that we regard as a one-parameter family of probability measures on
$L^p((0,1))$, $p\in [1,\infty)$. For $\ell$ fixed, by the representation~\eqref{mud}, the Laplace-Varadhan lemma implies that the family $\left(\mu_{\beta,\ell}\right)_{\beta>0}$ satisfies a large deviations principle with speed $\beta$ and rate function $\mc F_\ell$. In view of the variational convergence of the sequence of functionals $(F_\ell)$, we then deduce that $\mu_\beta$ satisfies a large deviation principle with speed $\beta$ and rate function $F$ as in~\eqref{F0} when the limit $\ell\to\infty$ is taken after the limit $\beta\to\infty$. We here show that this statement holds whenever $\ell_\beta$ grows sub-exponentially in $\beta$. On the other hand, as first discussed in \cite{Oea}, if $\ell_\beta$ grows exponentially in $\beta$, namely $\alpha>0$ in~\eqref{ipell}, then an entropic effect does modify the rate function. More precisely, the modified rate function $I\colon L^p((0,1))\to [0,+\infty]$ is given by
\begin{equation}
\label{Imod}
I(X)  := \begin{cases} (C_\mathrm{W}-\alpha)   |S(X)| &  \textrm{if } X\in BV( (0,1);\{-1,1\}) \\ +\infty  & \textrm{otherwise.}
\end{cases}
\end{equation}
In view of~\eqref{ipell} and the entropy in the interfaces location, this rate function reflects the heuristic that the probability of $n$ interfaces is roughly $\ell_\beta^n \rme^{- n \beta C_\mathrm{W}}$.

\begin{theorem}
\label{mainthmldp}
Assume~\eqref{ipell} and fix $p\in[1,\infty)$. The family $\left(\mu_\beta\right)_{\beta>0}$ of probability measures on $L^p((0,1))$ satisfies a large deviation principle with speed $\beta$ and the good rate function $I$ in~\eqref{Imod}. Namely, $I$ has compact level sets and for each closed $\mc C\subset L^p((0,1))$ and open $\mc O\subset L^p((0,1))$,
\[
\varlimsup_{\beta\to\infty} \frac 1\beta \log \mu_{\beta} (\mc C) \le - \inf_{X\in \mc C} I(X), \qquad \varliminf_{\beta\to\infty} \frac 1\beta \log \mu_{\beta} (\mc O) \ge - \inf_{X\in \mc O} I(X).
\]
\end{theorem}

The choice of the free boundary conditions for the $\phi^4_1$ measure is not particularly relevant for the validity of the large deviation principle. The arguments used in proof of Theorem~\ref{mainthmldp} cover directly the case in which $\mu_{\beta,\ell}$ is the restriction to the interval $[0,\ell]$ of the infinite volume $\phi^4_1$ measure. It is also possible to deduce the large deviations for $\mu_{\beta,\ell}^{a,b}$, the $\phi^4_1$ measure with the Dirichlet boundary conditions $X_0=a$, $X_\ell=b$, $a,b\in\bb R$; in this case the rate function depends however on the signs of $a$ and $b$. We do not pursue this issue here and refer to the analysis in \cite{Oea}, in which it is considered the case of Dobrushin boundary conditions specified by the choice $a=-1$, $b=1$.

\subsection*{Sharp asymptotics in the regime $\boldsymbol{1\ll \ell \ll \bar\ell}$}

The next result describes the sharp asymptotics of the $\mu_\beta$ probability of neighborhoods of the family of profiles with $n$ transitions between the pure phases. More precisely, we show that this probability behaves as if the number of transitions were a Poisson random variable with parameter $\ell/\bar \ell$. To this end, we fix a sequence $(\ell_\beta)_{\beta>0}$ satisfying
\begin{equation}
\label{ipells}
\lim_{\beta\to\infty} \beta^{-2} \ell_\beta =\infty \quad \textrm{ and } \quad \lim_{\beta\to\infty} \frac{\ell_\beta}{\bar\ell_\beta}= 0.
\end{equation}
The first requirement is technical and an arbitrary power law growth on $\ell_\beta$ should suffice. On the other hand, when $\ell_\beta$ grows only logarithmically in $\beta$ other phenomena become relevant, we refer to the analysis in \cite{BBB} for the case of Dobrushin boundary conditions.

We start by introducing the family of profiles with $n$ transitions. For $n=0$ we set $m^{\pm,0} = \pm 1$. For $n\in \bb N$ and $0< s_1<\dots < s_n<1$, let $m^{\pm, n}_{s_1, \dots,s_n} \colon (0,1) \to \{-1,1\}$ be the profile with $n$ jumps at times $s_1,\dots, s_n$ defined by
\begin{equation}
\label{mn}
m^{\pm,n}_{s_1, \dots,s_n} = \pm  (-1)^{k+1} ,  \textrm{ on }  [s_{k-1},s_k ),  \  k=1,\dots, n+1,
\end{equation}
where we understand  $s_0=0$ and $s_{n+1}=1$. Let finally $\mc M_n$ be the manifold of the profiles in $BV((0,1);\{-1,1\})$ with $n$ jumps, i.e.,
\begin{equation}
\label{Mn}
\mc M_n = \mc M_n^+\cup \mc M_n^-,  \qquad   \mc M_n^{\pm } = \bigcup_{0<s_1<\dots< s_n<1} \{m^{\pm, n}_{s_1, \dots, s_n}\},
\end{equation}
where we understand $\mc M_0 =  \{ m^{+,0}, m^{-,0} \} $. We finally recall that $\mc O^p_\rho (\mc B)$ denotes the $\rho$-neighborhood of $\mc B\subset L^p((0,1))$, the

\begin{theorem}
\label{theorem_sharp}
Assume~\eqref{ipells}, fix $p\in [1, \infty)$ and $n\ge 0$. Let also  $(\rho_{k,\beta})_{\beta>0}$, $k=1,2$, be sequences satisfying $\lim_{\beta} (\sqrt \beta \wedge \beta^{1/p})\rho_{k,\beta} = \infty$, $k=1,2$, and $\lim_\beta \rho_{1,\beta} = 0$. Then, as $\beta \to \infty$,
\[
\mu_\beta\big( \mc O^p_{\rho_{2,\beta}} (\mc M_{n}) \setminus \mc O^p_{\rho_{1, \beta}}  (\mc M_{n-1})\big) \sim \frac 1{n!} \bigg(\frac{\ell_\beta}{\bar\ell_\beta}\bigg)^n = \frac {\big(\ell_\beta A_\mathrm{W} \sqrt \beta \big)^n}{ 2^n n! } \rme^{- n \beta C_\mathrm{W}}.
\]
\end{theorem}

In the above statement, for $n=0$, it is understood that $\mc O^p_{\rho_{1,\beta}} (\mc M_{-1}) = \emptyset$. The condition $\rho_{1, \beta} \to 0$ is required to allow enough entropy on the interfaces location to pick the prefactor $\ell^n$. The condition $\rho_{k,\beta}  \gg \beta^{-1/2} \vee \beta^{-1/p} $, $k=1,2$, is required to accommodate for the typical fluctuations of the probability $\mu_\beta$ in $L^p((0,1))$, see Lemma~\ref{lemma_Pgrande} (ii).

As in the case of Theorem~\ref{mainthmldp}, the arguments used to deduce the sharp asymptotics cover directly the case in which $\mu_{\beta,\ell}$ is the restriction to the interval $[0,\ell]$ of the infinite volume $\phi^4_1$ measure. In contrast, the analysis cannot be applied to the probability $\mu_{\beta,\ell}^{a,b}$ with Dirichlet boundary conditions, as the crucial ingredient provided by Theorem~\ref{t:sht} cannot be applied. 

\subsection*{The regime $\boldsymbol{\ell \gg\bar\ell}$} 

To describe the asymptotic behavior of the $\phi^4_1$ measure in this regime, it will be convenient to regard the probability $\mu_{\beta,\ell}$ defined on the interval $[-\ell/2,\ell/2]$ rather than $[0,\ell]$. We further consider $\mu_{\beta,\ell}$ as a probability on $C(\bb R)$ by understanding $X_t=X_{-\ell/2}$ for $t< -\ell/2$ and $X_t=X_{\ell/2}$ for $t> \ell/2$.

Fix a sequence $(\ell_\beta)_{\beta>0}$ such that
\begin{equation}
\label{ipell3}
\lim_{\beta\to\infty} \frac{ \ell_\beta - \bar\ell_\beta }{\beta} = +\infty.
\end{equation}
We here consider the rescaled  $\phi^4_1$ measure defined by
\[
\bar \mu_\beta := \mu_{\beta,\ell_\beta}\circ \imath_{\bar\ell_\beta}^{-1} .
\]
For $p\in[1,\infty)$, we regard $(\bar \mu_\beta)_{\beta>0}$ as a sequence of probability measures on $L^p_\mathrm{loc}(\bb R)$. As we next state, this sequence converges weakly to some not trivial probability measure. In order to describe the limiting probability, we consider the continuous-time Markov chain with state space $\{-1, 1\}$ and generator 
\[
\bar L = \begin{pmatrix} -1  &  +1  \\  +1   &  -1 \end{pmatrix}.
\]
We denote by $\bar \mu$ the law of the stationary process associated to $\bar L$, regarded as a probability measure on $L^p_\mathrm{loc}(\bb R)$. Equivalently, given the family $\Sigma = \left(s_k\right)_{k\in\bb Z}$ satisfying $s_k< s_{k+1}$, $k\in \bb Z$, and $\lim_{k\to\pm \infty}s_k =\pm \infty$, let
$m^\pm_\Sigma$ be the element of $L^p_\mathrm{loc}(\bb R)$ defined by
\[
m^\pm_\Sigma  =\pm \sum_{k\in \bb Z} (\pm1)^k \id_{[s_k,s_{k+1})} .
\]
Denoting by ${\bar\mu}^\pm$ the probability on $L^p_\mathrm{loc}(\bb R)$ concentrated on $ m^\pm_\Sigma$ when $\Sigma$ is sampled according to a Poisson point process on $\bb R$ with intensity one, then $\bar \mu = \frac 12 \bar \mu^+ + \frac 12 \bar \mu^-$.

\begin{theorem}
\label{t:thm3}
Assume~\eqref{ipell3} and fix $p\in[1,\infty)$.  As $\beta\to \infty$ the sequence $(\bar \mu_\beta)_{\beta>0}$ of probability measures on $L^p_\mathrm{loc}(\bb R)$  converges weakly to $\bar\mu$.
\end{theorem}

In contrast to Theorem~\ref{theorem_sharp}, the choice of free boundary condition for $\mu_{\beta,\ell}$ is not relevant for Theorem~\ref{t:thm3}. The same limiting probability is obtained for any, $\ell$-independent, choice of the boundary conditions or directly for the infinite volume $\phi^4_1$ measure. We also note that in the critical case $\ell_\beta=\bar\ell_\beta$, Theorem~\ref{t:thm3} still holds with the following modifications. The probability $\bar\mu_\beta$ is regarded as a probability on $L^p((-1/2,1/2))$ and $\bar\mu$ is the marginal on the interval $(-1/2,1/2)$ of the probability described above.

As can be easily checked, the statement of Theorem~\ref{t:thm3} does not hold in the Skorohod topology. The convergence to the limiting measure $\bar \mu$ is a generic feature of bistable reversible diffusions in the small noise limit~\cite{Landim,PeletierSav,ReSe}. The choice of the $L^p$ topology on the space of paths appears however novel. In the context of the  one-dimensional Kac-Ising model, the statement analogous to Theorem~\ref{t:thm3} has been proven in~\cite{COP}, see also~\cite{SchonTan} for the one-dimensional nearest neighbors Ising model.

\subsection*{Outline of proofs and organization of the paper}

As well known \cite{SimonFI}, the infinite volume $\phi^4_1$ measure is the law of the stationary process associated to the one-dimensional reversible diffusion described by the stochastic differential equation 
\begin{equation}
\label{diff}
\rmd X_t = -\frac 1\beta \big(\log \psi_{0, \beta}\big)'(X_t) \rmd t + \sqrt{\beta^{-1}} \rmd w_t,
\end{equation} 
where $w$ is a standard Brownian on $\bb R$ and $\psi_{0, \beta}$ is the ground state of the quantum anharmonic oscillator with Hamiltonian $H_\beta =  - \frac{1}{2\beta} \frac{\rmd^2}{ \rmd x^2} + \beta W$. As we detail below, also the finite volume $\phi^4_1$ measure with free boundary conditions can be represented in terms of the solutions to~\eqref{diff}.

The present analysis is based upon sharp estimates for the metastable behavior of the solution to~\eqref{diff} in the small noise limit, which are the content of Section~\ref{sec:4}. As first exploited in~\cite{JMS}, these estimates can be deduced from the sharp semiclassical analysis of the Schr\"odinger operator $H_\beta$, here developed in Section~\ref{sec:3}, by applying the ground state transformation. In particular, as $\beta \to \infty$, the typical length $\bar \ell_\beta$ can be characterized as twice the inverse of the spectral gap of $H_\beta$. In contrast to the standard setting of reversible diffusions, here the potential in the Schr\"odinger operator does not depend on $\beta$, while the drift $- \beta^{-1}(\log \psi_{0, \beta})'$ depends on the semiclassical parameter and exhibits a discontinuity at the saddle point in the limit $\beta\to \infty$.

The proof of Theorem~\ref{mainthmldp}, given in Section \ref{sec:5}, will be achieved by combining the large deviation asymptotics of the solution to~\eqref{diff} over a fixed time window with the behavior of the discrete time Markov chain that is defined by sampling $X$ at the transition times between the pure phases. In this respect, some of the arguments are similar to those used in~\cite[Chapter 4]{FW}.

Theorems~\ref{theorem_sharp} and ~\ref{t:thm3} are proved in Section~\ref{sec:6} and \ref{sec:7}, respectively. For both of them the sharp asymptotics, that includes the prefactor of the transition time for the diffusion~\eqref{diff}, is relevant. More precisely, Theorem~\ref{t:thm3} essentially follows from the convergence in law of the suitably rescaled transition time to an exponential random variable. For Theorem~\ref{theorem_sharp}, it will be instead relevant to deduce the asymptotic probability of having $n$ transitions between the pure phases in the time interval $[0, \ell]$. As stated above, the latter point will be derived from the sharp semiclassical analysis of the quantum anharmonic oscillator with potential $W$ given in Section~\ref{sec:3}. This includes, as a novel point, a uniform control on the difference between the first two eigenfunctions in terms of the spectral gap, see Theorem~\ref{efsplitting}.

\section{Sharp semiclassical analysis of the anharmonic oscillator}  
\label{sec:3}

Within this section we denote by $L^p=L^p(\bb R)$, $p\in[1,\infty]$, and by $H^1=H^1(\bb R)$ the standard Lebesgue and Sobolev spaces on $\bb R$. The inner product in $L^2$ is denoted by $\langle\cdot,\cdot\rangle$. 

Let $H=H_\beta$ be the selfadjoint Schr\"odinger operator on $L^2$ which on its core $C^2_\mathrm{c}(\bb R)$ is given by
\begin{equation}
\label{He}
H_\beta  = -\frac {1} {2\beta} \frac{\rmd^2}{\rmd x^2}  + \beta W.
\end{equation}
Since $W$ has compact level sets, $H_\beta$ has purely discrete spectrum. We denote by $E_{k} = E_{k,\beta}$, $k \in \bb Z_+$, the corresponding eigenvalues enumerated in increasing order and by $\psi_k = \psi_{k, \beta}$, $k \in \bb Z_+$, a corresponding orthonormal basis of eigenfunctions. By the symmetry of $W$, $\psi_k(x) = (-1)^k \psi_k(-x)$.  In the sequel, we choose the ground state $\psi_0$ to be positive and $\psi_1$ to be positive on $\bb R_+$.

We set
\begin{equation}
\label{4.1}
U_\beta = -  \frac 1\beta \log \psi_{0,\beta}
\end{equation}
and observe that, as follows from a direct computation, it satisfies the Riccati equation
\begin{equation}
\label{riccati}
\frac 12 (U_\beta')^2 -\frac{1}{2\beta} U_\beta'' = W -  \frac 1 \beta E_{0,\beta}.
\end{equation}
Let also
\begin{equation}
\label{augdis}
U(x) =  \min \bigg\{  \bigg|\int_{-1}^x \!  \rmd  y \, \sqrt{2W(y)}  \bigg|, \bigg| \int_{1}^x  \!  \rmd  y \,  \sqrt{2W(y)}  \bigg| \bigg \}
\end{equation}
be the Agmon distance of $x$ from the two wells $-1$ and $1$. In particular, recalling~\eqref{tensup}, $C_\mathrm{W}= 2 U(0)$. For the special choice of the double-well potential~\eqref{W=}, $U(x) = \frac 13 |1-|x|^2|(2+|x|)$.

Let $\bar E_k:= (k +\frac 12) \sqrt{W''(\pm 1)}= 2k+1$, ${k\in \bb Z_+}$, be the eigenvalues of the harmonic oscillators centered in $\pm 1$ defined by the Hamiltonians $-\frac 1{2\beta} \frac{\rmd^2}{\rmd x^2} + \frac{\beta}2 W''(\pm 1) (x\mp1)^2$ and denote by
\[
g_{\pm} (x) = \Big(\beta\pi^{-1}\sqrt{W''(\pm 1)}\Big)^{1/4} \exp\Big\{ - \frac 12 \beta \sqrt{W''(\pm 1)} \, \big(x \mp 1\big)^2  \Big\}
\]
the corresponding ground states. Set finally
\[
g_0 = \frac{1}{Z_{0, \beta}} \frac{1}{\sqrt 2} \big( g_+ +   g_-  \big),
\qquad g_1 = \frac{1}{Z_{1, \beta}} \frac{ 1}{ \sqrt 2} \big( g_+  -   g_-  \big) , 
\]
where $Z_{k,\beta}$ is chosen so that $\|g_k\|_{L^2} = 1$, $k=0,1$. Note that $g_0$ and $g_1$ are respectively even and odd. Furthermore, $Z_{k,\beta} \sim 1$, $k=0,1$.

In the next statement we summarize the properties of $H_\beta$ that are relevant for the present analysis. Recall that the constant $A_\mathrm{W}$ has been introduced in~\eqref{AW}.

\begin{proposition}
\label{Prop_Semicl}
$~$
\begin{itemize}
\item[(i)]For each $k\in \bb Z_+$,
\[
\lim_{\beta \to \infty}   E_{k, \beta}      =  
\begin{cases}
\bar E_{k/2} =k+1   & \text{ if } k \text{ is even}, \\
\bar E_{(k-1)/2} =k   & \text{ if } k \text{ is odd}.
\end{cases}
\]
\item[(ii)] As $\beta \to \infty$,
\[
E_{1, \beta} - E_{0, \beta}
\sim A_\mathrm{W}  \sqrt \beta \,  \rme^{-\beta C_\mathrm{W}}.
\]
\item[(iii)] Uniformly on compact sets, $U_\beta \to U$ as $\beta \to \infty$.
\item[(iv)]There exist $R,\gamma\in (0,\infty)$ such that, for any
$\beta \ge 1$ and $|x| \ge R$,
\[
| \psi_k(x)| \le \rme^{- \beta \gamma |x|}, \qquad  k=0,1.
\]
\item[(v)]
There exists $C\in (0,\infty)$ such that, for any $\beta \ge 1$,
\[
\| \psi_k - g_k\|^2_{L^2}   \le C\beta^{-1}, \qquad  \| \psi_k - g_k\|^2_{L^\infty} \le C, \qquad  k=0,1.
\]
\end{itemize}
\end{proposition}

We refer to~\cite{HS1,Simon1} for the proof of item (i),
to~\cite{Harrell1,Harrell2,HelfferLNM} for (ii), to \cite{Simon2} for (iii) and (iv). For completeness, the proof of (v) will be detailed after a preliminary lemma.  

For $R >0$ we denote by $H^R = H^R_\beta$ the selfadjoint Schr\"odinger operator $- \frac{1}{2\beta} \frac{\rmd^2}{\rmd x^2} + \beta W$  on $L^2((-R, R))$ with Dirichlet boundary conditions at the endpoints. We understand $H^R= H$ when $R= \infty$.  We denote by $E_0^R = E_{0, \beta}^R$ the bottom of the spectrum of $H^R$ and, for $c\ge E_0^R$, by $\id_{[c, \infty)} (H^R)$ the spectral projector associated to $H^R$ on the interval $[c, \infty)$.

\begin{lemma} 
\label{Chebyshev}  
Fix $R\in (0, \infty]$.
\begin{itemize} 
\item[(i)] Let $0\le E \le E_0^R < c $ and $f$ be in the form domain of $H^R$. Then, for any $\beta > 0$,
\[   
\| \id_{[c, \infty)} (H^R) f\|^2_{L^2((-R, R))} \le \frac{1}{c-E}  \| (H^R-E)^{\frac 12}  f \|^2_{L^2((-R, R))}.  
\]
\item[(ii)]
Let $0\le E < c $ and $f$ be in the domain of $H^R$. Then, for any $\beta > 0$,
\[
\| \id_{[c, \infty)} (H^R) f\|^2_{L^\infty((-R, R))}  \le \frac{1}{c-E}  \left(2\beta + \frac{ 2\beta E + 1 }{c-E} \right)   \|(H^R-E) f\|^2_{L^2((-R, R))} .   
\]
\end{itemize}
\end{lemma}

\begin{proof}
The first statement follows directly from the spectral theorem and the Che\-byshev inequality. To prove (ii), let $g= \id_{[c, \infty)} (H^R) f $ and denote by $(P_\lambda)_{\lambda\ge 0}$ the projection-valued measure associated to $H^R$. Since $W\ge 0$,
\begin{align*}
& \| g\|^2 _{H_0^1((-R, R))} \le  2\beta \int_{-R}^R\! \rmd x\,  \left[ \frac{1}{2\beta} (g')^2 + (\beta W - E) g^2 \right] +   (2\beta E +1  ) \|g\|^2_{L^2((-R, R))} \\ & \qquad =  2\beta   \int_c^\infty (\lambda - E) \rmd \langle g, P_\lambda g\rangle_{L^2((-R, R))} +   (2\beta E +1  )  \int_c^\infty  \rmd \langle g, P_\lambda g\rangle_{L^2((-R, R))} \\ &  \qquad\le   \frac{2\beta}{c-E} \| (H^R-E) g\|^2_{L^2((-R, R))} +  \frac{2\beta E +1 }{(c-E)^2} \| (H^R-E) g\|^2_{L^2((-R, R))}  \\ &  \qquad \le \frac{2\beta}{c-E} \| (H^R-E) f\|^2_{L^2((-R, R))} +  \frac{2\beta E +1 }{(c-E)^2} \| (H^R-E) f\|^2_{L^2((-R, R))}.
\end{align*}
The conclusion follows by the Sobolev embedding. 
\end{proof}

\begin{proof}[Proof of~Proposition~\ref{Prop_Semicl} (v)]
Fix $k=1,2$ and set $R_k =   g_k - \lp g_k, \psi_k\rp \psi_k $ Observe that by symmetry $R_k$ is orthogonal to both $\psi_0$ and $\psi_1$. Further, since $\lp g_k, \psi_k \rp \ge 0$, 
\[
| \lp g_k, \psi_k \rp -1 | = \bigg|\frac{ \lp g_k, \psi_k \rp^2 -1}{\lp g_k, \psi_k \rp +  1 } \bigg| \le \|R_k\|_{L^2}^2.
\]
Hence
\[
\|g_k- \psi_k\|_{L^2}   \le      \|R_k\|_{L^2}   + \|R_k\|^2_{L^2}
\]
and
\begin{equation}
\label{3.65}
\| g_k -\psi_k\|_{L^\infty} \le \|R_k\|^2_{L^2}
\|\psi_k\|_{L^\infty} +   \|R_k\|_{L^\infty}. 
\end{equation}

To prove the first bound we apply Lemma~\ref{Chebyshev} (i) with $R=\infty$, $E= \bar E_0=1$, $f=g_k$ and $c\in (\bar E_0, \bar E_1)$ so that $R_k = \id_{[c, \infty)} (H) g_k $. We deduce 
\[   
\|R_k\|^2_{L^2} \le   \frac{1}{c-\bar E_0}  \|(H- \bar E_0)^{\frac 12} g_k\|_{L^2}^2. 
\]
By a direct computation, 
\[
\|(H- \bar E_0)^{\frac 12} g_{\pm} \|^2_{L^2} = \lp (H- \bar E_0) g_{\pm} , g_{\pm} \rp_{L^2}   = \beta \int \! \rmd x\, \big(\tfrac 12 x^4 \pm 2x^3 \big) \sqrt{\frac{2\beta}{\pi}} \rme^{-2\beta x^2}.
\]
It follows that  there is a constant $C$ independent of $\beta$ such that $ \|(H- \bar E_0)^{\frac 12} g_k \|^2_{L^2} \le C/\beta$, which concludes the proof of the first bound. 

To prove the second bound we apply Lemma~\ref{Chebyshev} (ii) again with $R=\infty$, $E= \bar E_0=1$, $f=g_k$  and $c\in (\bar E_0, \bar E_1)$. We deduce 
\[   
\|R_k\|^2_{L^\infty} \le  \frac{1}{c-\bar E_0}  \left(2\beta + \frac{ 2\beta \bar E_0 + 1 }{c-\bar E_0)} \right)   \|(H-\bar E_0) g_k\|^2_{L^2} . 
\]
By a direct computation,
\[
\|(H- \bar E_0)g_{\pm} \|^2_{L^2} = \beta^2 \int \! \rmd x\,
\big(\tfrac 12 x^4 \pm 2x^3 \big)^2 \sqrt{\frac{2\beta}{\pi}}
\rme^{-2\beta x^2} .
\]
It follows that there is a constant $C$ independent of $\beta$ such that  $ \|(H- \bar E_0) g_k \|^2_{L^2} \le C/\beta$. Further,  by the Sobolev embedding and Proposition~\ref{Prop_Semicl} (i),  $\|\psi_k\|_{L^\infty} \le C  \sqrt{\beta}$. We conclude recalling~\eqref{3.65}.
\end{proof}

The estimates stated below are readily deduced from Proposition~\ref{Prop_Semicl} (iv) and the first bound in Proposition~\ref{Prop_Semicl} (v).

\begin{corollary}
\label{t:3.2}
Let $I_{\pm}$ be neighborhoods of $\pm1$ such that $\mp1$ does not belong to the closure of  $I_\pm $. Then, for $k=0,1$, as
$\beta \to \infty$,
\begin{gather}
\label{p01aa}
\int_{ I_{+}}\!\rmd x\, \psi_{k,\beta} (x) \sim \left(\frac{\pi}{2\beta}\right)^{1/4}, \qquad \int_{I_{-}}\! \rmd x\, \psi_{k,\beta} (x) \sim  (-1)^k \left(\frac{\pi}{2\beta}\right)^{1/4},   \\ \label{p01cc} \pm \int_{ I_{\pm}}\!\rmd x\, \psi_{0,\beta} (x) \psi_{k,\beta} (x) \sim \frac 12.
\end{gather}
Moreover,
\begin{equation}
\label{p01bb} 
\int_{\bb R}\!\rmd x\, \psi_{0,\beta} (x) \sim  2 \left(\frac{\pi}{2\beta}\right)^{1/4}, \qquad \int_{\bb R}\!\rmd x\, \psi_{1,\beta}(x) =0. 
\end{equation}
\end{corollary}

The semigroup $\left(\rme^{-t (H-E_0)}\right)_{t\ge 0}$ admits a smooth kernel $g_t(x,y)= g^\beta_t(x,y)$, $t>0$, $x, y\in \bb R$, that can be represented as
\begin{equation}
\label{seriesg}
g_t(x,y) = \sum_{k=0}^\infty \rme^{- t (E_k - E_0)} \psi_k (x) \psi_k(y).
\end{equation}
Given $\ell$ and $T\in [0, \ell/2)$, let $\wp^{T} = \wp^{T}_{\beta, \ell} $ be the probability measure on $\bb R^2$ with density $\varrho^{T}=\varrho^{T}_{\beta,\ell}$ given by
\begin{equation}
\label{defwpT}
\varrho^{T} (x,y) = \frac 1{Z_\ell} \int\! \rmd a\, \rmd b  \,  g_T(a,x) g_{\ell-2T}(x,y) g_T(y,b), \quad Z_\ell = \int\! \rmd a\, \rmd b  \, g_{\ell}(a,b). 
\end{equation}
When $T=0$ we write $\wp$ for $\wp^{T}$ and $\varrho$ for $\varrho^T$ so that $\varrho(x,y) = Z_\ell^{-1} g_{\ell}(x,y)$. As discussed in \cite{SimonFI}, $\wp^{T}$ is the law of $(X_T, X_{\ell-T})$ when $X$ is sampled according to $\mu_{\beta, \ell}$. As we next prove, when $T\gg \beta$ the probability $\wp^{T}$ is well approximated by the probability $\bar\wp^{T}$ whose density  $\bar\varrho^{T}=\bar\varrho^{T}_{\beta,\ell}$ is given by
\begin{equation}
\label{defbarwpT}
\bar\varrho^T(x,y) =    \psi_0(x) g_{\ell-2T}(x,y)  \psi_0(y).
\end{equation}
Note that the probability $\bar\wp^T$ is the marginal at the times $T$ and $\ell-T$ of the law of the infinite volume $\phi^4_1$ measure.

Let also $\pi^T=\pi^T_\beta$ be the probability on $\bb R$ with density $q^T=q^T_\beta$ given by
\begin{equation}
\label{qT}
q^T(x)  =    \frac{1}{\psi_0(0)} g_T( 0, x) \psi_0(x),
\end{equation}
which is the marginal at time $T$ of the $\phi^4_1$ measure on the half-line $[0, \infty)$ with boundary  condition $X_0=0$. As we next prove, when $T\gg \beta$ the probability $\pi^{T}$ is well approximated by the probability $\pi=\pi_\beta$ with density $\psi_0^2$, that is the single time marginal of the infinite volume $\phi^4_1$ measure.

In the next statement we denote by $\|\cdot\|_\mathrm{TV}$ the total
variation norm on the space of finite signed measures.

\begin{proposition}
\label{sharpwp}
$~$
\begin{itemize}
\item[(i)]
There exists $\ell_0>0$ such that if $\varliminf_\beta \ell_\beta \ge \ell_0$ then the sequence of probability measures $( \wp_{\beta, \ell_\beta})_{\beta>0}$ is exponentially tight as $\beta\to \infty$. Namely, there exists a sequence of compacts $K_R\subset\subset\bb R^2$ such that
\[
\lim_{R\to \infty} \varlimsup_{\beta \to \infty}   \frac 1\beta \log  \wp_{\beta, \ell_\beta} ( K_R^{\mathrm{c}}) =   -\infty.
\]
\item[(ii)]
Assume that $T_\beta$ satisfies $\lim_\beta \beta^{-1}T_{\beta} =\infty$ and $\ell_\beta$ be such that $\ell_\beta> 2T_\beta+1$. Then
\[
\lim_{\beta \to \infty} \frac 1\beta \log \|\wp^{T_\beta}_{\beta,\ell_\beta} - \bar{\wp}^{T_\beta}_{\beta,\ell_\beta} \|_{\mathrm{TV}} = - \infty.    
\]
\item[(iii)]
Assume that $T_\beta$ satisfies $\lim_\beta \beta^{-1}T_{\beta} =\infty$. Then
\[
\lim_{\beta \to \infty} \frac 1\beta \log  \| \pi^{T_\beta}_\beta  - \pi_\beta\|_{\mathrm{TV}} = - \infty. 
\]  
\end{itemize}
\end{proposition}

We remark that the exponential tightness in (i) actually holds for any $\ell_0>0$ but we refrain from proving this stronger statement as it is not needed in the present analysis. We also emphasize that the super-exponential bounds (ii) and (iii) depend on the invariance of the boundary conditions with respect to the map  $X\mapsto -X $.

We premise the following lemma based on the Thompson-Goldstein
inequality~\cite[Theorem 9.2]{SimonFI}.

\begin{lemma}
\label{traceformula2}
There exists a constant $C$ such that, for any $T\ge 1$ and $\beta$ large enough,
\[
\sum_{k=2}^\infty \rme^{- T(E_k- E_0)}    \left( 1 +
\|\psi_k\|^2_{L^1}    \right)    \le C \rme^{-  T} . 
\]
\end{lemma}

\begin{proof}
By the Thompson-Goldstein inequality, for each $\lambda>0$,
\begin{equation} 
\label{TGI} 
\sum_{k=0}^\infty \rme^{- \lambda E_k} \le \frac{\beta}{2\pi} \int_{\bb R^2}\! \rmd x\, \rmd p\, \rme^{-\lambda \beta (\frac 12 p^2 + W(x))} \le C ,
\end{equation}
for a suitable constant $C= C(\lambda)$.   On the other hand, for a suitable constant $C$, whose numerical value may change from line to line, 
\begin{align*}
& \sum_{k=2}^{\infty} \rme^{- T (E_k-E_0)} \|\psi_k\|^2_{L^1} = \sum_{k=2}^{\infty} \rme^{- T (E_k-E_0)} \bigg(\int_{\bb R} \!\rmd x\, \frac{ \sqrt{1+W(x)}}{\sqrt{1+W(x)}} |\psi_{k}(x)| \bigg)^2 \\ & \qquad \le C \sum_{k=2}^{\infty} \rme^{-T (E_k-E_0)} \int_{\bb R} \!\rmd x\, [1+W(x)] \, |\psi_{k}(x)|^2 \\ & \qquad \le C \sum_{k=2}^{\infty} \rme^{-T (E_k- E_0)} \big(1 + \beta^{-1} E_k\big) \le C \sum_{k=2}^{\infty} \rme^{-(T-\beta^{-1})E_k + T E_0} \\ & \qquad \le C \rme^{ E_2} \rme^{- T (E_2 - E_0)} \sum_{k=2}^{\infty} \rme^{-(1-\beta^{-1})E_k}.
\end{align*}
The statement follows then from~\eqref{TGI} and Proposition~\ref{Prop_Semicl} (i).
\end{proof}

\begin{proof}[Proof of Proposition~\ref{sharpwp}]
$~$

\noindent
\emph{Proof of (i).}  By~\eqref{seriesg} and~\eqref{p01bb}, 
\[
\varliminf_{\beta\to\infty}  \frac 1\beta \log  Z_{\ell_\beta}  \ge  \varliminf_{\beta\to\infty}  \frac 1\beta \log  \|\psi_0\|_{L^1}^2   =   0.
\]
It is therefore enough to show
\begin{equation} 
\label{exptightg}
\lim_{R\to \infty} \varlimsup_{\beta \to \infty}   \frac 1\beta \log   \int_{K_R^{\mathrm{c}}} \! \rmd x\,  \rmd y\,  g_{\ell_\beta}( x, y) =   -\infty.
\end{equation}
We claim that there exists a constant $C$ such that, for every Borel subset $A \subset \bb R^2$ and $\beta \ge 1$,
\begin{equation} 
\label{hyperc}
\int_{A}\! \rmd x\,  \rmd y\,   g_{\ell_\beta}( x, y)  \le   C    \bigg( \int_A \! \rmd x\, \rmd y\, \big[ \psi_0(x) \psi_0 (y) \big]^{2/3}\bigg)^{3/4}.
\end{equation}
This claim, together with Proposition~\ref{Prop_Semicl} (iv) yields~\eqref{exptightg}. In order to prove~\eqref{hyperc} we use that the semigroup $\left(\rme^{-t H_\beta}\right)_{t>0}$ is {\it intrinsically hypercontractive}, uniformly in $\beta$. More precisely, by~\cite[Theorem 6.5]{DavSim}, there exist $T, C \in (0, \infty)$ such that, for any $\beta \ge 1$,
\[
\| \rme^{T L_\beta} f\|_{L^4(\pi_\beta)} \le C   \|  f\|_{L^2(\pi_\beta)} , 
\]
where $L_\beta =   - \frac{1}{\psi_0}(H_\beta - E_0) \psi_0 $. In particular,
\[
\bigg\| \frac{\psi_k}{\sqrt{\psi_0}}  \bigg\|_{L^4}   =  \bigg\| \frac{\psi_k}{\psi_0} \bigg\|_{L^4(\pi_\beta)}  \le C  \rme^{T (E_k - E_0)} .
\]
This last estimate together with H\"older inequality, yields
\begin{gather*}
\int_{A}\! \rmd x\, \rmd y \,  g_{\ell_\beta}( x, y)  \le   \sum_{k =0}^\infty  \rme^{ - \ell_\beta (E_k - E_0)}     \int_A\! \rmd x\, \rmd y \, \frac{| \psi_k (x)|}{\sqrt{\psi_0(x)}}  \frac{| \psi_k (y)|}{\sqrt{\psi_0(y)}}  \, \sqrt{\psi_0(x)}   \sqrt{\psi_0(y)} \\ \le   C^2  \sum_{k =0}^\infty   \rme^{ - (\ell_\beta -  2T) (E_k - E_0)}    \bigg( \int_A\! \rmd x\, \rmd y \,  \big[ \psi_0(x) \psi_0 (y) \big]^{2/3}\bigg)^{3/4}      .  
\end{gather*}
Choosing $\ell_0 > 2T$, the conclusion follows from the Thompson-Goldstein inequality~\eqref{TGI} and Proposition~\ref{Prop_Semicl} (i).

\smallskip
For the proof of (ii) and (iii) we first observe that, in view of~\eqref{seriesg},~\eqref{p01bb}, and Lemma~\ref{traceformula2}, the normalization in~\eqref{defwpT} satisfies
\begin{equation}
\label{stimaZpw}
\lim_{\beta\to \infty} \frac 1\beta \log \big| Z_{{\ell_\beta}} - \|\psi_0\|^2_{L^1} \big| =  -\infty
\end{equation}
whenever $\lim_\beta \beta^{-1} {\ell_\beta} =\infty$.

\smallskip
\noindent
\emph{Proof of (ii).}
Set
\begin{equation}
\label{splitseries2}
\mc E_{T_\beta}(x,y) = \int_{\bb R^2}\!  \rmd a \, \rmd b \, g_{T_\beta} (a,x) g_{T_\beta} (y,b) - \| \psi_0 \|^2_{L^1} \, \psi_0(x) \psi_0(y) ,
\end{equation}
so that
\[
Z_{\ell_\beta} \varrho^{T_\beta}(x,y) = \|\psi_0\|^2_{L^1} \bar{\varrho}^{T_\beta} (x, y) + \mc
E_{T_\beta}(x,y) g_{{\ell_\beta}-2{T_\beta}}(x,y) .
\]
By~\eqref{seriesg} and~\eqref{p01bb} we have
\begin{align*} 
|\mc E_{T_\beta}(x,y)| & \le C\beta^{-1/4} \sum_{k=2}^\infty \rme^{- T_\beta (E_k-E_0) } \|\psi_k\|_{L^1} ( |\psi_k(x)| \psi_0(y) + |\psi_k(y)| \psi_0(x) ) \\ & \quad + \sum_{k, j=2}^\infty \rme^{- T_\beta(E_k + E_j -2E_0) } \, \|\psi_k\|_{L^1} \|\psi_j\|_{L^1} |\psi_k(x)| |\psi_j(y)|. 
\end{align*}
Arguing as in the proof of Lemma~\ref{traceformula2}, $\|\psi_k\|_{L^1} \le C\sqrt{1+\beta E_k}$ so that, by Proposition~\ref{Prop_Semicl} (i),  after elementary computations, we deduce that there exists a constant $C$ such that for every $x,y \in \bb R^2$ and $\beta$ large enough,
\begin{align} 
\label{stimasuE}
|\mc E_{T_\beta}(x,y)| & \le C \rme^{- {T_\beta}} \sum_{k=2}^\infty \rme^{- E_k } \, ( |\psi_k(x)| \psi_0(y) + |\psi_k(y)| \psi_0(x) ) \nonumber \\ & \quad + C \rme^{- 2 {T_\beta}} \sum_{k, j=2}^\infty \rme^{- ( E_k + E_j) } \, |\psi_k(x)| |\psi_j(y)|.  
\end{align}
By Lemma~\ref{traceformula2} and the Cauchy-Schwartz inequality we conclude that
\[
\lim_{\beta\to \infty} \frac 1\beta \log  \int\! \rmd x\, \rmd y\, |\mc E_{T_\beta}(x,y)|g_{{\ell_\beta}-2{T_\beta}}(x,y) = -  \infty,
\]
which, together with~\eqref{p01bb} and~\eqref{stimaZpw} implies the statement.

\smallskip
\noindent
\emph{Proof of (iii).}
Since $\psi_1(0) =0$, from~\eqref{seriesg} we deduce
\[
\begin{split}
\| \pi^{T_\beta}_\beta - \pi_\beta\|_{\mathrm{TV}} &\le \frac 12 \frac{1}{\psi_0(0)}\sum_{k =2}^\infty \rme^{-{T_\beta}(E_k-E_0)} |\psi_k(0)| \|\psi_k\|_{L^1} \\ &\le \frac 14 \frac{1}{\psi_0(0)}\sum_{k =2}^\infty \rme^{-{T_\beta}(E_k-E_0)} \big( |\psi_k(0)|^2 + \|\psi_k\|^2_{L^1} \big).
\end{split}
\]
Using Proposition~\ref{Prop_Semicl} (iii) and Lemma~\ref{traceformula2}, to conclude the proof of the statement it is enough to show
\begin{equation} 
\label{pinco}
\lim_{\beta\to \infty} \frac 1\beta \log \sum_{k =2}^\infty \rme^{-{T_\beta}(E_k-E_0)} |\psi_k(0)|^2 = - \infty .
\end{equation}
To this end, observe that, by Sobolev embedding and~\eqref{He}, there
exists a constant $C<\infty$ such that for every $k$
\begin{equation} 
\label{Sobolev}
\|\psi_k\|_{\infty}^2 \le C \|\psi_k\|^2_{H^1} \le C( 1 + 2\beta E_k ) 
\end{equation}
and~\eqref{pinco} follows by arguing as in the proof of Lemma~\ref{traceformula2}.
\end{proof}

The last statement of this section provides a global control on  $|\psi_1|-\psi_0$  in terms of the spectral gap $E_1-E_0$. This appears to be a novel point in the semiclassical analysis of $H_\beta$ and it is the crucial ingredient for obtaining the sharp asymptotics in Theorem~\ref{theorem_sharp}. 

\begin{theorem} 
\label{efsplitting}
There exists $C\in(0,\infty)$ such that for every $\beta\ge 1$ 
\[
\sup_{x\ge 0}  \, \psi_0(x) \big|\psi_0(x) - \psi_1(x)\big| \le C \beta^{\frac 32} (E_1-E_0) . 
\]
\end{theorem}

\begin{proof}
The proof is based on the construction of a test function $f_1$ which provides a good global approximation to $\psi_1$. To this end, let $\bar\chi\in C^\infty(\bb R)$ be such that $0\le \bar\chi\le 1$, $ \bar\chi (x)= 1 $ on $(-\infty, -2]$, and $\bar\chi (x)= 0 $ on $[-1, +\infty)$.  Further, for $\delta\in (0,\frac 12)$, set $\chi(x) = \chi_\delta(x) = \bar \chi (\delta^{-1}(|x|-1))$.  Note that $\chi_\delta =1$ on $[-1+2\delta, 1-2\delta]$ and $\chi_\delta = 0$ on the complement of $[-1+\delta, 1-\delta]$. Define
\[
\theta(x) := \kappa \int_0^x\! \rmd y\, \frac{\chi(y)}{\psi_0^2(y)}, \qquad   \kappa = \kappa_\beta := \bigg(\int_0^1\! \rmd y\, \frac{\chi(y)}{\psi_0^2(y)} \bigg)^{-1} .  
\]
In particular  $\theta(x) = 1 $ for $x\ge 1 $.  Set finally
\[
f_1  :=  \frac 1Z \theta \psi_0, \qquad Z   =   Z_\beta  := \bigg(\int\! \rmd x\,  \theta^2  \psi^2_0 \bigg)^{\frac 12} .    
\]
Note that $f_1$ is antisymmetric and normalized in $L^2$. Further, by Proposition~\ref{Prop_Semicl} (iv), for every $a\in (0,1]$, \begin{equation}
\label{kappa}   
\frac{1}{\kappa} \sim  \int_0^a \! \rmd y\, \frac{\chi(y)}{\psi_0^2(y)}  \sim   \int_0^a \! \rmd y\, \frac{1}{\psi_0^2(y)} \quad  \text{ and }   \quad  \lim_{\beta\to \infty} \frac{1}{\beta} \log \kappa =  - C_\mathrm{W} .      
\end{equation}

\smallskip
\noindent
\emph{Step 1.}  There exists a constant $C$ such that, for every $\beta\ge 1$,
\[   
|Z-1| \le  C\kappa   \quad  \text{ and }  \quad   \sup_{x\ge 0} \psi_0(x) |\psi_0(x) - f_1(x)| \le  C  \sqrt\beta \kappa .   
\]

Since $E_0 \to 1$ and $\psi_0$ satisfies $ \frac {1} {2\beta} \psi_0''   - \beta W \psi_0 = - E_0 \psi_0$ we deduce that there exists a constant $c>0$ depending only on  $W$ such that $\psi_0'(x) \ge 0$ for $x\in [0, 1- c \beta^{-\frac 12}]$ and $\beta$ large enough.  Hence, choosing  $\delta=\delta_\beta = c\beta^{-\frac 12} $,   for  $x \in [0,1]$, 
\begin{equation} 
\label{theta}
0 \le 1- \theta(x)   =   \kappa \int_x^1\!\rmd y\, \frac{\chi(y)}{\psi_0^2(y)} \le   \frac{\kappa}{\psi_0^2(x)} . 
\end{equation}
Therefore,
\[
0 \le 1  - Z^2   =   2 \int_{0}^1 \! \rmd x\, [1 -  \theta^2(x) ]  \psi_0^2 (x)    \le   4 \kappa,   
\]
which proves the first claim. To prove the second, we write
\[ 
\psi_0(x) |\psi_0(x) - f_1(x)|  \le   \psi_0^2(x) | 1- \theta(x)| +   \psi_0^2(x) |1 - Z^{-1}| . 
\]
We then conclude by using~\eqref{theta}, Proposition~\ref{Prop_Semicl} (v), and the bound on $|Z-1|$. 

\smallskip
\noindent
\emph{Step 2.} There exists a constant $C>0$ such that for any $\beta$  large enough
\[    
\int_{\bb R}\!\rmd x\,  f_1 (H_\beta - E_0) f_1   \sim \beta^{-1} \kappa \quad  \text{ and  }   \quad \int_{\bb R} \! \rmd x\, |(H_\beta - E_0) f_1 |^2 \le    C \beta^{-2} \kappa^2 .  
\]

To prove the first claim, note that by the ground state transformation, 
\[   
\int_{\bb R}   \!  f_1 (H_\beta - E_0) f_1  \rmd x    = \frac{1}{2\beta} \int_{\bb R} \! \rmd x\, \bigg|\Big(\frac{f_1}{\psi_0} \Big)' \bigg|^2 \psi_0^2  = \frac{ \kappa^2 }{ \beta Z^2}  \int_0^1  \! \rmd x\, \frac{\chi^2 }{ \psi_0^2  } \sim   \beta^{-1}  \kappa, 
\]
where we used the first statement in~\eqref{kappa} and, as follows from Step 1, $Z\sim 1$. For the second claim, again by the ground state transformation, 
\begin{gather*} 
\int_{\bb R} \! \rmd x\,  \big|(H_\beta - E_0) f_1 \big|^2 =  \int_{\bb R} \! \rmd x\,   \bigg| \frac{1}{2\beta} \frac{1}{\psi_0^2} \bigg[ \psi_0^2 \Big(\frac{f_1}{\psi_0}\Big)' \bigg]' \bigg|^2   \psi_0^2   =    \frac{\kappa^2}{Z^2 (2\beta)^2 }    \int_{\bb R}   \!  \rmd x\, \frac{|\chi'|^2}{\psi_0^2}.
\end{gather*}
By the definition of $\chi_{\delta}$, the choice of $\delta = c\beta^{-\frac 12}$,  and the monotonicity of $\psi_0$ on the interval $[0, 1-\delta_\beta]$, we deduce that there exists a constant $C>0$ such that, for any $\beta$ large enough,
\[   
\int_{\bb R} \!\rmd x \,  |(H_\beta - E_0) f_1 |^2  \le C \frac{\kappa^2}{Z^2 \beta^{3/2} }  \frac{1}{\psi^2_0 (1- \delta_\beta)}  \le C \frac{\kappa^2}{ \beta^{2} } ,
\]
where we used the second bound in Proposition~\ref{Prop_Semicl} (v) and $Z\sim 1$ as proven in Step 1.

\smallskip
\noindent
\emph{Step 3.}    As $\beta \to \infty$,  $\beta (E_1- E_0) \sim \kappa  $. 

In view of the orthogonality $\langle f_1, \psi_0\rangle=0$ and the first statement of Step 2, the upper bound follows from the Riesz variational principle. The lower bound follows from Step 2, $E_2 -E_0 \to 2 $, as stated in Proposition~\ref{Prop_Semicl} (i), and the Temple inequality, see e.g.~\cite{Harrell1} or \cite[Prop.~5.1]{DG}. 

\smallskip
\noindent
{\it Conclusion.}  We write 
\[ 
\psi_0 (x) |\psi_1(x) - \psi_0(x)| \le \psi_0 (x) |\psi_0(x) - f_1(x)| + \psi_0 (x) |\psi_1(x) - f_1(x)| .  
\] 
The first term on the right hand side is bounded by using Steps 1 and 3. By Lemma~\ref{Chebyshev} (ii) with $R=\infty$, $E= E_0$,  $c \in ( \bar E_0, \bar E_1) $, and Steps 2-3,
\[ 
\| f_1 - \langle f_1, \psi_1 \rangle \psi_1 \|_{L^\infty} \le C \sqrt \beta (E_1- E_0), 
\] 
for a suitable constant $C$. Furthermore, since $\langle f_1, \psi_1 \rangle \ge 0$, Lemma~\ref{Chebyshev} (i) with $R, E, c$ as before, and again  Steps 2-3 imply
\[ 
|\langle f_1, \psi_1 \rangle -1|   \le  |\langle f_1, \psi_1 \rangle^ 2 -1|  \le   \| f_1 - \langle f_1, \psi_1 \rangle \psi_1 \|^2_{L^2}    \le  C  (E_1-E_0).    
\]
We conclude the proof by using Proposition~\ref{Prop_Semicl} (v) to bound $\|\psi_k\|_{\infty}$ for $k=0,1$.
\end{proof}

\section{Small  noise analysis of the diffusion process}
\label{sec:4}

The present analysis of the $\phi^4_1$ measure is based on a representation of the probability $\mu_{\beta,\ell}$ in terms of the diffusion~\eqref{diff} that we here describe. Given $\ell>0$ and $x,y\in \bb R$, denote by $\bb P^{\beta, \ell}_{x,y}$ the law of the solution to~\eqref{diff} starting at time $t=0$ from $x$ conditioned to $X_\ell=y$. The probability $\bb P^{\beta, \ell}_{x,y}$ can be identified with the $\phi^4_1$ measure on $[0,\ell]$ with boundary conditions $X_0=x$ and $X_\ell =y$. Since the marginal distribution of $\mu_{\beta, \ell}$ at the endpoints is  the probability $\wp_{\beta,\ell}$ with density given in~\eqref{defwpT} with $T=0$, we obtain the following representation.  
\begin{equation}
\label{rme}
\mu_{\beta,\ell} (\rmd X) = \wp_{\beta,\ell}(\rmd X_0\,\rmd X_\ell) \, \bb P^{\beta,\ell}_{X_0,X_\ell}(\rmd X). 
\end{equation}

Given $x\in \bb R$ we denote by $\bb P_x^\beta$ the law of the solution to
\eqref{diff} with initial condition $x$, that we regard as a probability on $C([0,\infty))$. By using the result on the semiclassical analysis of the anharmonic oscillator deduced in Section~\ref{sec:3}, in this section we derive the asymptotic behavior of this probability in the limit $\beta\to\infty$. In the next sections, we will then achieve the proofs of the main results by exploiting the representation~\eqref{rme}.

A peculiar feature of the diffusion~\eqref{diff} is that, according to Proposition~\ref{Prop_Semicl} (iii), the potential $U_\beta$ in~\eqref{4.1} exhibits a corner at the local maximum $x=0$ in the limit $\beta\to\infty$. This will be an important feature in the analysis of its metastable behavior, implying in particular that in the transition between the pure phases $\pm 1$ essentially no time is spent in a neighborhood of zero.

We first show that with probability super-exponentially close to one the solution to~\eqref{diff} is bounded on exponentially large intervals. 

\begin{lemma}
\label{p:1}
Let $(\ell_\beta)_{\beta>0}$ be such that $\varlimsup_\beta \beta^{-1} \log\ell_\beta < \infty$. 
Then for each $L>0$
\begin{equation}
\label{6.31}
\lim_{R\to\infty} \varlimsup_{\beta\to\infty}\sup_{|x|\le L}
\frac 1\beta \log\bb P_x^{\beta} \Big( \sup_{t\in [0,\ell_\beta]} |X_t| > R
\Big) =-\infty.
\end{equation}
\end{lemma} 

\begin{proof}
By monotonicity we can assume $\varliminf_\beta \beta^{-1}\ell_\beta = \infty$. For $R>0$ let $\tau_R$ be the hitting time of $(-R,R)^\mathrm c$.  We observe that $u(t,x) := \bb P^\beta_x (\tau_R > t)$ solves 
\[    
\begin{cases}  \partial_t u  =  \frac{1}{2\beta} \partial^2_x u  - U'_\beta(x)   \partial_x  u,  &  \text{on } \bb (0, \infty)  \times  (-R, R),   \\ u(0, x) = 1,   &  x>0, \\ u(t,\pm R)  = 0,   & t>0. 
\end{cases}
\]
Let now $L^R = L_\beta^R$ be the  selfadjoint realization on $L^2((-R,R), \pi_\beta)$ of  $\frac{1}{2\beta} \frac{\rmd^2}{\rmd x^2}  -   U'_\beta(x)   \frac{\rmd}{\rmd x}$ with zero Dirichlet  boundary condition in $x=\pm R$. Here, we understand that $\pi_\beta$ is the sub-probability  on $(-R,R)$ given by  $\psi_0^2 \rmd x$.  Analogously, we let $H^R = H_\beta ^R$  be the  selfadjoint realization on $L^2((-R,R))$ of the Schr\"odinger operator 
$-\frac{1}{2\beta} \frac{\rmd^2}{\rmd x^2} +  \beta W $ with zero Dirichlet  boundary condition in $x=\pm R$. Let also $(E_k^R)_{k\ge 0}$ be the eigenvalues of $H^R$ and let $(\psi_k^R)_{k\ge 0}$ be a corresponding  orthonormal basis of eigenfunctions. By the ground state transformation,  
$L^R = - \frac{1}{\psi_0} (H^R - E_0) \psi_0 $, so that
\[
u(t,x)    =  \frac{1}{\psi_0(x)} \sum_{ k =0}^\infty \rme^{- t(E_k^R - E_0)}  \langle \psi_0 , \psi^R_k \rangle_{L^2((-R, R))} \psi_k^R (x).
\]
Hence, by the antisymmetry of $\psi_1^R$, 
\begin{align*}   
& \bb P_x^{\beta} \Big( \sup_{t\in [0,\ell_\beta]} |X_t| > R \Big) = 1 - u(\ell_\beta,x) \\ &  \qquad \qquad = \frac{1}{\psi_0(x)}  \left(  \psi_0(x) -   \rme^{- \ell_\beta(E_0^R- E_0)}  \langle \psi_0 , \psi^R_0 \rangle_{L^2((-R, R))} \psi_0^R (x)  \right)   \\   & \qquad\qquad \quad  +    \frac{1}{\psi_0(x)} \sum_{ k = 2}^\infty \rme^{- \ell_\beta(E_k^R - E_0)}  \langle \psi_0 , \psi^R_k \rangle_{L^2((-R, R))} \psi_k^R (x).    
\end{align*}
By Proposition~\ref{Prop_Semicl} (iii) and~\eqref{augdis}, for each $L>0$, 
\[    
\lim_{\beta} \sup_{|x| \le L}  \frac 1\beta \log \frac{1}{\psi_0(x)}  =  \max_{ |x|\le L} U(x) < \infty. 
\]
The proof is thus completed provided we show that for each $L>0$
\begin{align}
\label{p1} 
& \lim_{R\to \infty} \varlimsup_{\beta \to \infty} \frac 1\beta \log (E_0^R - E_0) = -\infty, \\ \label{p2}   & \lim_{R\to \infty} \varlimsup_{\beta \to \infty} \sup_{|x| \le L} \frac 1\beta \log \big|\psi_0(x) - \langle \psi_0, \psi_0^R \rangle_{L^2((-R, R))} \psi_0^R(x) \big| = -\infty, \\ \label{p3}  & \lim_{\beta \to \infty} \sup_{|x| \le L} \frac 1\beta \log \sum_{k\ge 2} \rme^{-\ell_\beta(E_k^R - E_0) }  |\psi_k^R(x)| = -\infty .
\end{align}

To prove~\eqref{p1}, let $\chi_R \in C_\mathrm{c}^\infty((-R,R); [0,1])$ be symmetric and such that $\chi_R(x) = 1$ for $|x| \le R-1$ and $|\chi_R'(x)|\le 2$. Observe that 
\[   
\langle \chi_R \psi_0,  (H^R - E_0)  \chi_R \psi_0   \rangle  =   \frac{1}{2\beta} \int\! \rmd x\, |\chi_R' |^2 \psi_0^2.
\]
The claim then follows from Proposition~\ref{Prop_Semicl} (iv) and the variational characterization of the bottom of the spectrum.

In order to prove~\eqref{p2}, we set $f^R = \psi_0\chi_R$ with $\chi_R$ as above. Choosing $c \in (E^R_0,E^R_2 )$ and using the antisymmetry of $\psi_1^R$,
\[   
g^R :=  \id_{[c, \infty)} (H^R) f^R  = f^R - \langle f^R, \psi_0^R \rangle_{L^2((-R, R))} \psi_0^R . 
\]
We then have, for $R$ large enough, 
\begin{gather*}
\sup_{|x|\le L } \big|\psi_0(x) - \langle \psi_0, \psi_0^R \rangle_{L^2((-R, R))} \psi_0^R(x) \big|    =   \sup_{|x|\le L } \big|f^R(x) - \langle \psi_0, \psi_0^R \rangle_{L^2((-R, R))} \psi_0^R(x) \big| \\   \le    \sup_{|x|\le R}  |g^R(x)| +   |\langle \psi_0 (\chi_R -1), \psi_0^R \rangle_{L^2((-R, R))}|  \sup_{|x| \le R} \psi_0^R (x)   \\ \le   \sup_{|x|\le R}  |g^R(x)| +     2\|\psi_0^R\|^2_{H_0^1((-R, R))} \int_{R-1}^{\infty}\!\rmd x \, \psi_0(x) .
\end{gather*}
Observing that $\|\psi_0^R\|^2_{H_0^1((-R, R))} \le 2\beta E_0^R +1$ and using that $E^R_0 \to \bar E_0$ as $\beta \to \infty$,  the second term is super-exponentially small in view of Proposition~\ref{Prop_Semicl} (iv). To bound the first term we observe that $E^R_2 \to \bar E_1$ and apply Lemma~\ref{Chebyshev} (ii) with $E= E_0$ and $c\in (\bar E_0, \bar E_1 )$. We deduce that there exists a constant $C$ such that for any $\beta$ large enough
\[ 
\sup_{|x|\le R}  |g^R(x)|^2    \le  C \beta \|(H^R-E_0) f^R\|^2_{L^2((-R, R))}.
\]
The norm on the right hand side reads
\begin{align*}
& \| (H^R-E_0) f^R\|^2_{L^2((-R, R))}    =  \int_{-R}^R\! \rmd x  \, \bigg[     -\frac{1}{2\beta} (\chi_R \psi_0)''  +   (\beta W- E_0) \chi_R \psi_0    \bigg]^2   \\ &  = \int_{-R}^R \! \rmd x  \,  \bigg( \frac{1}{2\beta} \chi_R''\psi_0   +   \frac 1\beta \chi'_R \psi'_0 \bigg)^2 \le   \frac{C}{\beta^2 }   \bigg( \int_{-R}^R \! \rmd x  (\chi''_R)^2 \psi_0^2  +   \int_{-R}^R \! \rmd x\,  (\chi'_R)^2  (\psi'_0) ^2   \bigg). 
\end{align*}
The first term in the right-hand side is directly estimated by using Proposition~\ref{Prop_Semicl} (iv). For the second one, we integrate by parts, obtaining
\begin{gather*}
\int_{-R}^R \! \rmd x\,  (\chi'_R)^2  (\psi'_0) ^2  = -   \int_{-R}^R \! \rmd x\, \bigg[  ((\chi'_R)^2)'  \psi'_0 \psi_0  +  (\chi'_R)^2  \psi''_0\psi_0 \bigg]  \\ =   \int_{-R}^R \! \rmd x \bigg[  \frac 12 ((\chi'_R)^2)''  \psi^2_0  -  (\chi'_R)^2  \psi''_0\psi_0 \bigg] .      
\end{gather*}
We conclude by using $- \frac{1}{2\beta}\psi_0'' + \beta W  \psi_0 = E_0 \psi_0$ and again Proposition~\ref{Prop_Semicl} (iv).  

To prove~\eqref{p3}, we first observe that by the Sobolev embedding
\begin{gather*}
\sup_{|x|\le L} |\psi_k^R(x)|^2 \le \sup_{|x|\le R} |\psi_k^R(x)|^2 \le \|\psi_k^R\|_{H^1}^2 \le   \big( 1 + 2\beta (E_k^R- E_0 )\big),
\end{gather*}
so that
\[
\sup_{|x|\le L} \sum_{k= 2}^\infty \rme^{-  \ell_\beta  (E_k^R- E_0)} |\psi_k^R(x)|  \le   \sum_{k= 2}^\infty  \rme^{-\ell_\beta(E_k^R - E_0) + 2 \beta (E_k^R - E_0)}.
\]
Recalling we have assumed $\varliminf_\beta \beta^{-1}\ell_\beta = \infty$, by Weyl's Theorem \cite{HJ}, the min-max principle,  $E_k^R \ge  E_k$, and Proposition~\ref{Prop_Semicl} (iii) we conclude by arguing as in the proof of Lemma~\ref{traceformula2}.
\end{proof}

We next show that, with $\bb P_x^\beta$ probability super-exponentially close to one, the time average of $t\mapsto W(X_t)$ is arbitrarily small.  For the analysis of the sharp asymptotics, we need a quantitative statement that requires condition~\eqref{ipells}.

\vskip1.5cm
\begin{lemma}
\label{lemma_Pgrande}
$~$
\begin{itemize}
\item[(i)]
Fix $\delta>0$ and a sequence $(\ell_\beta)_{\beta>0}$
such that $\lim_\beta \ell_\beta =\infty$.  Then, for any $L>0$,
\[
\lim_{\beta \to\infty}  \frac{1}{\beta} \sup_{|x| \le L } \log\bb P^{\beta}_x \bigg(\frac 1{\ell_\beta} \int_0^{\ell_\beta}\!\rmd t\, W(X_t)  > \delta \bigg) = -\infty.
\]
\item[(ii)]
Fix $\gamma>\bar E_0 =1$ and a sequence $(\ell_\beta)_{\beta>0}$
such that $\lim_\beta\beta^{-2}\ell_\beta =\infty$.  Then, for any
$L>0$, 
\[
\lim_{\beta \to\infty}  \frac{1}{\beta} \sup_{|x| \le L} \log\bb P^{\beta}_x \bigg(\frac 1{\ell_\beta} \int_0^{\ell_\beta}\!\rmd t\, W(X_t)  > \frac\gamma\beta \bigg) = -\infty.
\]  
\end{itemize}
\end{lemma}

\begin{proof}
By It\^o's formula,
\begin{equation}
\label{martingala}
U_\beta (X_t) - U_\beta(x) + \int_0^t\!\rmd r\, \Big( U'_\beta(X_r)^2 - \frac{1}{2\beta} U''_\beta(X_r) \Big) = M^\beta_t,  
\end{equation}
where $(M^\beta_t)_{t\ge 0}$ is a continuous $\bb P^{\beta}_x$-martingale with quadratic variation
\[
\langle M^\beta\rangle_t = \frac 1\beta \int_0^t\!\rmd r\,
U'_\beta(X_r)^2. 
\]
We claim that there exists $C\in (0,\infty)$ and for any $L>0$ a constant $C_L$ such that for any $\beta$ large enough 
\[
\frac{1}{2\beta} U_\beta'' \le \frac 12 (U_\beta')^2 + C,
\qquad U_\beta \ge - C,    \qquad     \max_{|x|\le L } U_\beta(x) \le C_L.
\]
Indeed, the first bound follows directly from Proposition~\ref{Prop_Semicl} (i), Riccati's equation~\eqref{riccati}, and the positivity of $W$. The second one follows from items (iii) and (iv) in Proposition~\ref{Prop_Semicl} and the last statement again from Proposition~\ref{Prop_Semicl} (iii).

Equation~\eqref{martingala} thus implies that for each $x\in [-L, L]$ with $\bb P_x^{\beta}$ probability one we have 
\[
\frac \beta2 \langle M^\beta \rangle_t \le M^\beta_t + C (t +1) + C_L,    \quad t\ge 0.
\] 
By using the martingale inequality in~\cite[Lemma 2]{Mauro}, we deduce that for any $\eta,\ell>0$  
\begin{equation}
\label{usom}
\sup_{|x| \le L} \bb P^{\beta}_x \bigg(\sup_{t\in[0,\ell]} M^\beta_t > \eta \ell \bigg) \le \exp\Big\{ - \frac{\beta \eta^2\ell}{4(\eta +C)+4(C+C_L)\ell^{-1}} \Big\}. 
\end{equation}

We next observe that by~\eqref{riccati} and~\eqref{martingala}
\[
\begin{split}   
& \int_0^\ell\!\rmd t\, W(X_t) = \int_0^\ell\!\rmd t\, \Big( \frac 12 U_\beta'(X_t)^2 - \frac {1}{2\beta} U_\beta''(X_t) + \frac 1\beta E_{0,\beta} \Big) \\ & \qquad \le M^\beta_\ell - \big( U_\beta(X_\ell) - U_\beta(x)\big)
+ \frac 1\beta E_{0,\beta} \le M^\beta_\ell + C+ C_L + \frac 1\beta E_{0,\beta} \ell.
\end{split}
\]
Recalling that by Proposition~\ref{Prop_Semicl} (i) $ E_{0,\beta}\to \bar E_0$, the proof of both statements is achieved by combining the previous displayed bound with~\eqref{usom}. 
\end{proof}

The next statement will be used in the proof of the lower bound in Theorem~\ref{mainthmldp}.

\begin{lemma} 
\label{Propsemigruppo}
Fix $\delta\in (0,1)$ and set  $I^\pm_\delta := (\pm1-\delta, \pm 1+\delta)$. There exists $\gamma>0$ such that 
\[
\lim_{\beta \to \infty }\bb P_y^\beta (X_t  \in I_\delta^\pm) = 1  \quad \text{ uniformly for $y\in I_\delta^\pm$ and $t\in [\gamma \beta, \ell_\beta]$.}
\]
\end{lemma}
\begin{proof}
Using the antisymmetry of $\psi_{1, \beta}$ we have,
\[
\begin{split}
& \bb P_y^\beta (X_t  \in I_\delta^\pm) = \frac{1}{\psi_{0,\beta}(y)} \int_{I_\delta^\pm}\!\rmd x\, g(x,y;t) \psi_{0,\beta}(x) \\ & \qquad = \int_{I_\delta^\pm}\!\rmd x\, \psi_{0,\beta}(x)^2 + \rme^{-t(E_{1,\beta} -E_{0,\beta})} \frac{ | \psi_{1,\beta}(y)|}{ \psi_{0,\beta}(y) }
\int_{I_\delta^\pm}\!\rmd x\, \psi_{0,\beta}(x) |\psi_{1,\beta}(x)| + R_\beta(y,t), 
\end{split}
\]
where
\[
R_\beta(y,t)  = \sum_{k=2}^\infty \rme^{-t(E_{k, \beta} -E_{0,\beta})} \frac{ \psi_{k,\beta}(y)}{ \psi_{0,\beta}(y) }\int_{I_\delta^\pm}\!\rmd x\, \psi_{0,\beta}(x) \psi_{k,\beta}(x)
\]
and 
\[   
\frac{| \psi_{1,\beta}(y)|}{ \psi_{0,\beta}(y) }   =       1 +  \frac{  \psi_{0, \beta}(y) ( |  \psi_{1,\beta}(y)| -  \psi_{0,\beta}(y)) }{\psi_{0, \beta}^2 (y)}.
\]
By Theorem~\ref{efsplitting} and Proposition~\ref{Prop_Semicl} (iii),
\[   
\sup_{y\in I_\delta^{\pm}} \bigg|\frac{  \psi_{0, \beta}(y) ( |  \psi_{1,\beta}(y)| -  \psi_{0,\beta}(y)) }{\psi_{0, \beta}^2 (y)}  \bigg|  \le  C \beta (E_{1, \beta} - E_{0, \beta})  \exp \Big\{ 2 \beta   \max_{y\in I_\delta^\pm} U(y) \Big\}, 
\]
which vanishes as $\beta\to \infty$. Moreover, by Proposition~\ref{Prop_Semicl} (ii),   $\rme^{-t(E_{1,\beta} -E_{0,\beta})} \to 1$ for $t\le \ell_\beta$ as $\beta \to \infty$ and, by~\eqref{p01cc}, 
\[   
\lim_{\beta\to \infty} \int_{I_\delta^\pm}\!\rmd x\, \psi_{0,\beta}(x)^2   = \lim_{\beta\to \infty}   \int_{I_\delta^\pm}\!\rmd x\, \psi_{0,\beta}(x) |\psi_{1,\beta}(x)|  = \frac 12.   
\]
Finally, in view of Proposition~\ref{Prop_Semicl} (iii) and~\eqref{Sobolev}, arguing as in Lemma~\ref{traceformula2}, 
\begin{gather*}    
\max_{y\in I_\delta^{\pm}} |R_\beta(y,t)|  \le   C \rme^{C_\delta \beta}   \sum_{k = 2}^\infty (1+ \beta E_{k, \beta}) \rme^{- t (E_{k, \beta} - E_{0, \beta})}  \\  \le C \exp \Big\{ 2 \beta   \max_{y\in I_\delta^\pm} U(y) \Big\}     \rme^{ - (t-\beta) (E_{2, \beta} - E_{0, \beta}) } ,
\end{gather*}
which vanishes provided $\gamma$ is large enough. 
\end{proof} 

For $T>0$  let $\bb P_x^{\beta,T}$ be the marginal of $\bb P_x^\beta$ on $C([0,T])$. In view of the singularity of the potential in the limit $\beta\to\infty$, the diffusion~\eqref{diff} does not fit in the standard Friedlin-Wentzell scheme \cite{FW}. Nonetheless, as we next prove, $\bb P_x^{\beta,T}$ satisfies a large deviation principle with rate function $I_T^x\colon C([0,T])\to[0,\infty]$ given by
\begin{equation}
\label{ratefuncI}
I_T^x(X) =
\begin{cases}
\displaystyle{ \frac 12 \int_0^T\!\!\rmd t\, \big[\dot X_t^2  + U'(X_t)^2\big] + U(X_T) - U(X_0)} & \textrm{if } X\in  H^1_x([0,T]), \\ +\infty& \textrm{otherwise,}
\end{cases}
\end{equation}
where
\[
H^1_x([0,T]) := \Big\{ X\in C([0,T])\colon \int_0^T\!\rmd t\, {\dot{X}}^2_t <\infty,\;\; X_0=x\Big\}
\]
and $U$ has been defined in~\eqref{augdis}. We emphasize that $I_T^x$ is well defined because $(U')^2$ is continuous. Note also that, informally, $I_T^x$ is equal to the standard, but ill-defined, Friedlin-Wentzell rate function $(1/2)\int_0^T\!\rmd t\,\big[\dot X_t - U'(X_t)\big]^2$.

\begin{proposition} 
\label{LDPdiff}
Fix $T>0$. The sequence $(\bb P_x^{\beta,T})_{\beta>0}$ satisfies, uniformly for $x$ in compact sets, a large deviation principle with speed $\beta$ and rate function $I_T^x$. Namely, for each $x\in \bb R$, each sequence $x_\beta\to x$, each closed $\mc C\subset C([0,T])$, and each open $\mc O\subset C([0,T])$,
\[
\varlimsup_{\beta \to \infty} \frac 1\beta \log \bb P_{x_\beta}^{\beta,T}(\mc C) \le  - \inf_{X\in\mc C} I_T^x(X) , \qquad \varliminf_{\beta \to \infty} \frac 1\beta \log \bb P_{x_\beta}^{\beta,T}(\mc O) \ge  - \inf_{X\in\mc O} I_T^x(X).      
\]
\end{proposition}

\begin{proof}
Let $\bb P_x^{0,\beta,T}$ be the law of $x + \beta^{-1/2}\,w$ on $C([0,T])$. By the Schilder theorem the sequence $(\bb P_x^{0,\beta,T})_{\beta>0}$ satisfies, uniformly for $x$ in compact sets,  a large deviation principle with speed $\beta$ and rate function
\[
J_T^x(X) = \begin{cases} \displaystyle{ \frac 12 \int_0^T\!\!\rmd t\, {\dot X}_t^2 } & \textrm{if } X\in  H^1_x([0,T]), \\ +\infty& \textrm{otherwise.}
\end{cases}
\]

Let  $L=L_\beta =  \frac{1}{2\beta} \frac{\rmd^2}{\rmd x^2} -   U_\beta'  \frac{\rmd }{\rmd x} $ be the generator of~\eqref{diff} and denote by $p_t^\beta (x,y)$, $t>0$, $x, y\in \bb R$, the kernel of  the semigroup generated by $L$. Recalling that $g_t^\beta(x, y)$ is the kernel of the semigroup generated by $- (H_\beta - E_{0, \beta})$,  the ground state transformation yields  $p_t^\beta (x,y) =   \psi_{0, \beta } (x)^{-1}  g_t^\beta (x,y) \psi_{0, \beta} (y) $. By the Feynman-Kac representation for the kernel $g^\beta_t(x, y)$ we deduce 
\begin{equation}
\label{rnqp}
\frac{\rmd\bb P_x^{\beta,T}}{\rmd\bb P_x^{0, \beta,T}} =  \exp\big\{- \beta  \Phi^x_\beta\big\}, 
\end{equation}
in which 
\[
\Phi^x_\beta (X) = \int_0^T\!\rmd t\, \Big[W(X_t) -  \frac 1 \beta E_{0,\beta} \Big] + U_\beta(X_T)  - U_\beta(x). 
\]
By items (i) and (iii) in Proposition~\ref{Prop_Semicl}, $\Phi^{x}_\beta$ converges uniformly on compact subsets of $C([0, T])$
to
\[
\Phi^x (X) := \int_0^T\!\rmd t\, W(X_t) + U(X_T) - U(x) = \frac 12
\int_0^T\!\rmd t\, U'(X_t)^2  + U(X_T) - U(x), 
\]
uniformly for $x$ in compact subsets of $\bb R$. 

The proof is now completed by applying the corollary of the Laplace-Varadhan principle stated in~\cite[Ex. 4.3.11]{DZ}. Observe indeed that the latter can be generalized to the present setting in which $\Phi^x_\beta$ depends on $\beta$, but converges uniformly in compacts. It remains to show the tail condition~\cite[4.3.2]{DZ}, that in view of~\eqref{rnqp} can be restated as
\begin{equation}
\label{exptight}
\lim_{R\to \infty} \varlimsup_{\beta\to \infty}  \sup_{|x| \le L}\frac{1}{\beta} \log \bb P_x^{\beta,T} \big( \|X\|_\infty> R \big) = -\infty, \quad  L>0, 
\end{equation} 
and therefore follows directly from Lemma~\ref{p:1}.
\end{proof}

The next statement provides the sharp asymptotics of the hitting time of zero for the stationary process associated to~\eqref{diff}. It is the key ingredient to derive the sharp asymptotics of the $\phi^4_1$ measure described in Theorem~\ref{theorem_sharp}. Recall that $\pi_\beta (\rmd x) = \psi^2_0 \rmd x $ is the invariant measure of the diffusion~\eqref{diff}. We denote by $\bb P^\beta_{\pi_\beta}$ the law of the corresponding stationary process.

\begin{theorem}
\label{t:sht}
Let $\tau$ be the hitting time of zero and  $(T_{k,\beta})_{\beta>0}$, $k=0,1$,  be sequences such that $\lim_\beta \beta^{-3/2} T_{0,\beta} =\infty$ and $\lim_\beta \bar\ell_\beta^{-1} T_{1,\beta}=0$.  Then, uniformly for $t\in[T_{0,\beta},T_{1,\beta}]$,
\[
\bb P^\beta_{\pi_\beta} (  \tau < t ) \sim   \frac {2\,t}{\bar\ell_\beta} = {t A_\mathrm{W} \sqrt \beta } \rme^{-\beta C_\mathrm{W}}. 
\] 
\end{theorem}

\begin{proof}
By symmetry, 
\[
\bb P_{\pi_\beta}^\beta (\tau< t )  =  1  -  2\int_0^\infty\!\rmd x\, \psi_{0}^2(x) u (t, x), 
\]
where $u(t,x) := \bb P^\beta_x (\tau > t)$ solves 
\[    
\begin{cases}  \partial_t u  =  \frac{1}{2\beta} \partial^2_x u - U'_\beta    \partial_x  u, &  \text{on } \bb (0, \infty)  \times  (0, \infty),   \\ u(0, x) = 1,   &  x>0, \\ u(t,0)  = 0,   & t>0. 
\end{cases}
\]
Denote by $L^0_\beta$ the selfadjoint realization on $L^2(\bb R_+, 2\pi_\beta)$ of  $\frac{1}{2\beta} \frac{\rmd^2}{\rmd x^2} -   U_\beta'  \frac{\rmd }{\rmd x} $ with zero Dirichlet  boundary condition in $x=0$. By symmetry and the ground state transformation, the eigenvalues of $L^0_\beta$ are $ - \lambda_{k} = E_{2k-1, \beta} - E_{0, \beta}$, $k\ge 1$, and the corresponding  normalized eigenfunctions are $f_k =    \psi_{2k-1}/ \psi_0$, $k\ge 1$. Hence
\[
u(t, x) = \sum_{k=1}^\infty e^{- t \lambda_{k} }  \langle  f_k, 1\rangle_{L^2(\bb R_+, 2 \pi_\beta)} f_k(x).
\]
By Cauchy-Schwartz inequality and Lemma~\ref{traceformula2} there exists a constant $C$ such that, for $\beta$ large enough,
\[
\sum_{k=2}^\infty \rme^{- t \lambda_{k} } | \langle  f_k ,1\rangle_{L^2(\bb R_+, 2 \pi_\beta)} | \int_0^\infty\!\rmd x\,  | f_k(x)| \psi_{0}^2(x) \le  C \rme^{- t },  
\]
so that
\begin{equation} 
\label{rest_Dirichlet}
\bigg| \bb P^\beta_{\pi_\beta} (  \tau < t )    -  \bigg( 1   - 2\rme^{- t \lambda_{1} }  \langle  f_1, 1\rangle_{L^2(\bb R_+, 2 \pi_\beta)}  \int_0^\infty\!\rmd x\, f_1(x)  \psi_{0}^2(x)\bigg) \bigg| \le C \rme^{-   t }  .    
\end{equation}
On the other hand, 
\begin{gather*}   
1 - 2 \rme^{- t \lambda_{1} }  \langle  f_1, 1\rangle_{L^2(\bb R_+, 2 \pi_\beta)}  \int_0^\infty\!\rmd x\,  f_1(x)  \psi_{0}^2(x)  \\  = 1 - 4 \rme^{- t ( E_{1, \beta} - E_{0, \beta} ) }   \bigg( \int_0^\infty\!\rmd x\,  \psi_1(x) \psi_0(x) \bigg)^2  \\ =  1 - 4 \rme^{- t ( E_{1, \beta} - E_{0, \beta} ) }   \bigg( \frac 1 2  +  \int_0^\infty\!\rmd x\,  [\psi_1(x) - \psi_0(x)] \psi_0(x) \bigg)^2 .   
\end{gather*}
We conclude by using items (ii) and (iv) in Proposition~\ref{Prop_Semicl}, and Theorem~\ref{efsplitting}. 
\end{proof}

We finally show that under the measure $\bb P^\beta_{\pi_\beta}$ the random variable  $\tau /\bar\ell_\beta $ converges in law to an exponential  random variable with parameter $2$. 

\begin{theorem}
\label{t:sht2}
Let $\tau$ the hitting time of zero.  Then, for each $t \ge 0$,
\[
\lim_{\beta \to \infty}  \bb P^\beta_{\pi_\beta} \Big(  \frac{\tau}{ \bar \ell_\beta}  >  t  \Big) = \rme^{-2t} .
\]   
\end{theorem}

\begin{proof} 
We use the notation introduced in the proof of Theorem~\ref{t:sht}. As follows from~\eqref{rest_Dirichlet}, for $t>0$,
\[
\lim_{\beta \to \infty}  \bigg| \bb P^\beta_{\pi_\beta} (  \tau >  t \bar\ell_\beta ) -  2\rme^{- t \bar\ell_\beta  \lambda_{1} }  \langle  f_1, 1\rangle_{L^2(\bb R_+, 2 \pi_\beta)}  \int_0^\infty\! \rmd x\, f_1(x)  \psi_{0}^2(x)  \bigg) \bigg| =  0.    
\]
By~\eqref{barl} and Proposition~\ref{Prop_Semicl} (ii), $\bar\ell_\beta \lambda_1 =   \bar\ell_\beta (E_{1, \beta}- E_{0, \beta}) \to 2 $. Moreover, by using Proposition~\ref{Prop_Semicl} (iv) and Theorem~\ref{efsplitting},
\[  
\lim_{\beta \to \infty} 2 \langle  f_1, 1\rangle_{L^2(\bb R_+, 2 \pi_\beta)}  \int_0^\infty\!\rmd x\, f_1(x)  \psi_{0}^2(x) =   \lim_{\beta \to \infty} 4  \bigg( \int_0^\infty\!\rmd x\, \psi_1(x) \psi_0(x) \bigg)^2   = 1 ,    
\]
which concludes  the proof. 
\end{proof}

\section{Large deviation principle}
\label{sec:5}

In this section we complete the proof of Theorem~\ref{mainthmldp} by following the classical strategy of large deviation estimates. Namely, we first show that the family $(\mu_\beta)_{\beta>0}$, as defined in \eqref{1.10}, is exponentially tight, then we prove the upper bound on compacts, and finally derive the lower bound for neighborhoods of elements $X$ such that $I(X)$, as defined in~\eqref{Imod}, is finite. Some of the  bounds here obtained are in fact weaker than the sharp bound in Theorem~\ref{t:sht}, but they rely on the assumption~\eqref{ipell} rather than~\eqref{ipells}.  Within this section we fix once for all $p\in[1,\infty)$.

\subsection*{Super-exponential estimates}

Recall that $\imath_\beta$ denotes both the dilation $(0,\ell_\beta) \ni t\mapsto s=t/\ell_\beta \in (0,1)$ and its lift to functions, i.e., $(\imath_{\beta} X )(s) = X(\ell_\beta s )$, $s\in (0,1)$. According to the notation in Section~\ref{sec:4}, $\bb P_x^{\beta,\ell_\beta}$ is the law of the diffusion process satisfying~\eqref{diff} on $[0,\ell_\beta]$ with initial condition $x\in \bb R$, while $\mu_{\beta, \ell_\beta}$ is the $\phi^4_1$ measure on the interval $[0, \ell_\beta]$ with free boundary conditions on the end points. We regard both these measures as probabilities on $C([0,\ell_\beta])$.  We first show that events which have super-exponentially small probability with respect to $\bb P_x^{\beta,\ell_\beta}$ have also super-exponentially small probability with respect to $\mu_{\beta,\ell_\beta}$.

\begin{lemma}
\label{t:5.1}
Let $(\mc A_{\beta,k})_{k\ge 1}$ be a family of Borel sets in  $C([0,\ell_\beta])$ such that for each $L>0$
\[
\lim_{k\to\infty}\varlimsup_{\beta\to\infty} \sup_{|x|\le L} \frac1\beta \log \bb P_x^{\beta,\ell_\beta}\big( \mc A_{\beta,k} \big) =-\infty.
\]
Then
\[
\lim_{k\to\infty}\varlimsup_{\beta\to\infty} \frac1\beta \log \mu_{\beta,\ell_\beta}\big( \mc A_{\beta,k} \big) =-\infty.
\]
\end{lemma}

\begin{proof}
By the representation~\eqref{rme} and Proposition~\ref{sharpwp} (i), there exists $C_L \to \infty$ as $L\to \infty$ such that
\[
\mu_{\beta,\ell_\beta}\big( \mc A_{\beta,k} \big) \le  \int_{[-L,L]^2}\! \wp_\beta(\rmd x ,\rmd y)\, \bb P^{\beta,\ell_\beta}_{x,y}\big( \mc A_{\beta,k} \big) + \rme^{-\beta C_L}.
\]
By the ground state transformation
\begin{equation}
\label{repwp}
\wp_\beta(\rmd x\, \rmd y) = \frac{1}{Z_\beta} \frac{\psi_{0}(x)} {\psi_{0}(y)} \, \bb P_x^{\beta,\ell_\beta} (X_{\ell_\beta}\in \rmd y) \, \rmd x ,   
\end{equation}
with $Z_\beta=\int\!\rmd a\, \rmd b\, g_{\ell_\beta}(a,b)$. We thus deduce
\[
\mu_{\beta,\ell_\beta} \big( \mc A_{\beta,k} \big) \le \frac{1}{Z_\beta}  \sup_{|y|\le L} \frac{1}{\psi_{0}(y)} \sup_{|x| \le L} \bb P^{\beta,\ell_\beta}_{x} \big( \mc A_{\beta,k} \big) \, \|\psi_{0}\|_{L^1} + \rme^{-\beta C_L}.  
\]
The statement now follows from~\eqref{stimaZpw}, Proposition~\ref{Prop_Semicl} (iii), and Corollary~\ref{t:3.2}.
\end{proof}

As we next show, the marginal of the probability $\mu_{\beta,\ell_\beta}$ on $C([T,\ell_\beta-T])$ is super-exponentially close to $\bb P^\beta_{\pi_\beta}$ provided $T\gg \beta$. For the application to the sharp asymptotics, we also need an analogous statement for the probability $\bb P^\beta_{0}$, the law of the diffusion~\eqref{diff} starting from zero.  We emphasize that these estimates depend on the symmetry with respect to the map $X\mapsto -X$.

\begin{lemma}
\label{lemma_staz0}
Let $(T_\beta)_{\beta>0}$, $(\ell_\beta)_{\beta>0}$ be sequences such that $\beta^{-1}T_\beta \to \infty$ and  $\ell_\beta> 2 T_\beta$. Then
\[
\begin{split}
&\lim_{\beta\to \infty} \frac 1\beta \log \big| \mu_{\beta, \ell_\beta} (\mc B)
-   \bb P^\beta_{\pi_\beta} (\mc B) \big|=-\infty, \qquad \mc B\in  \sigma(\{ X_t\}_{t\in [T_\beta, \ell_\beta - T_\beta]}) ,  \\ &\lim_{\beta\to \infty} \frac 1\beta \log \big| \bb P^\beta_{0}(\mc B) -   \bb P^\beta_{\pi_\beta} (\mc B) \big|=-\infty, \qquad  \ \ \ \mc B\in  \sigma(\{ X_t\}_{t\ge T_\beta}). 
\end{split}
\]
\end{lemma}

\begin{proof}
For notation convenience, we set $\ell= \ell_\beta$ and $T=T_\beta$.
Recalling that $\bar{\wp}^{T} = \bar{\wp}^{T}_{\beta, \ell} $ is the
probability measure on $\bb R^2$ with density $\bar{\varrho}^{T}$
defined in~\eqref{defbarwpT}, we claim that
\[
\bb P^\beta_{\pi_\beta} (\mc B) = \int\! \bar{\wp}^{T}  (\rmd x \, \rmd y)\,  \bb P^{\beta, \ell-2T}_{x,y}\big(\theta_{-T}\mc B\big),       
\]
where $\theta_t$ is the time shift $(\theta_t X)_{t'}  = X_{t'- t}$. Indeed, by the ground state transformation, $\bb P_x^\beta(X_t\in \rmd y) = \psi_0(x)^{-1} g_t(x,y) \psi_0(y) \rmd y$. By the Markov property and the explicit form \eqref{defbarwpT} we deduce the claim.

On the other hand, recalling that $\wp^{T} = \wp^{T}_{\beta, \ell} $ is the probability measure on $\bb R^2$ with density $\varrho^{T}$ defined in~\eqref{defwpT}, 
\[
\mu_{\beta, \ell} (\mc B) = \int\!\wp^{T}  (\rmd x\, \rmd y)\, \bb P^{\beta, \ell-2T}_{x,y} \big(\theta_{-T}\mc B\big).
\]
The first statement thus follows from Proposition~\ref{sharpwp} (ii).

To prove the second statement, recall that $\pi^T$ is the probability on $\bb R$ with density defined in~\eqref{qT}. In particular, again by the ground state transformation and the Markov property,
\[
\bb P_0^\beta (\mc B) = \int \!\pi^T(\rmd x) \,\bb P_x^\beta\big(\theta_{-T}\mc B\big).    
\]
The second statement thus follows from Proposition~\ref{sharpwp} (iii).
\end{proof}

\subsection*{Exponential tightness} 

Within the general strategy of~\cite{FW}, we next introduce a discrete time Markov chain that allows to reduce the analysis of the asymptotic behavior of $\mu_{\beta, \ell_\beta}$ to that of its marginal on sub-intervals of $[0, \ell_\beta]$ of fixed size. Given $\rho=(\rho_1, \rho_2)$ with $0<\rho_1< \rho_2<1$ and $X\in C([0, \ell_\beta])$, we recursively define the sequence of stopping times $(\sigma_k)_{k\in \bb Z_+}$ by
\begin{equation} 
\label{sigmak}
\begin{split}
& \sigma_{2k} =  \inf \{ t\in [\sigma_{2k-1}, \ell_\beta] \colon X_t \in \{-1+\rho_1, 1-\rho_1\} \} \wedge \ell_\beta, \\  & \sigma_{2k+1} =  \inf \{ t\in [\sigma_{2k}, \ell_\beta] \colon X_t \in \{-1+\rho_2,1 -\rho_2\} \}
\wedge \ell_\beta,
\end{split}
\end{equation}
where it is understood $\sigma_{-1} = 0 $.  Moreover, let $\mc N_{\beta} =  \mc N_{\beta, \rho} = \sup \{ k\ge 0 \colon \sigma_k < \ell_\beta \}$ so that $\sigma_k = \ell_\beta$ for $k > \mc N_\beta$. We then set $Y_k = X_{\sigma_k}$ and consider  the family  $(Y_k)_{k=0}^{\mc N_\beta}$ taking values in the set $\{-1+\rho_1, -1+\rho_2, 1-\rho_2, 1-\rho_1\}$. 

By  the strong Markov property, when $X$ is sampled according to $\bb P_x^{\beta, \ell_\beta}$, the family $(Y_k)$ is a homogeneous discrete time Markov chain with transition probability given by 
\[      
\begin{pmatrix} 0 &  1 & 0&   0  \\ 1-p &   0 & p & 0  \\ 0 & p & 0 & 1- p \\  0 & 0 & 1 & 0     \end{pmatrix},   
\]
where, denoting by $\tau_z$ the hitting time of $z$,  
\[
p = p_{\rho, \beta} = \bb P_{-1+\rho_2}^\beta ( \tau_{1-\rho_1} < \tau_{-1+\rho_1})  =  \bb P_{1-\rho_2}^\beta ( \tau_{-1+\rho_1} < \tau_{1-\rho_1}).
\]
We next show that the number of jumps of this chain is bounded by  $\ell_\beta$ with probability super-exponentially close to one.

\begin{lemma}
\label{superexp}
For any $\rho=(\rho_1, \rho_2)$ with $0<\rho_1< \rho_2<1$,
\[
\varlimsup_{\beta\to\infty} \sup_{x\in \bb R} \frac 1\beta \log \bb P_x^{\beta, \ell_\beta} (\mc N_{\rho, \beta} >  \ell_\beta) = - \infty.
\]
\end{lemma}

\begin{proof}
By the exponential Chebyshev inequality, for every $\gamma>0$, 
\begin{align*}
&  \bb P_x^{\beta, \ell_\beta}( \mc N_{\rho, \beta} > \ell_\beta)   =\bb P_x^{\beta, \ell_\beta} \Big( \sum_{k=0}^{\lfloor \ell_\beta \rfloor} (\sigma_k -\sigma_{k-1}) \le \ell_\beta \Big)  \\ &\le  \bb P_x^{\beta, \ell_\beta} \Big( \sum_{k=0}^{\lfloor \ell_\beta/2 \rfloor -1} (\sigma_{2k+1} -\sigma_{2k}) \le \ell_\beta\Big)      \le \bb E_x^{\beta, \ell_\beta} \exp\Big( \gamma \ell_\beta - \gamma \sum_{k=0}^{\lfloor \ell_\beta/2 \rfloor -1}   (\sigma_{2k+1} -\sigma_{2k}) \Big) \\  &  =  \exp \Big(\gamma \ell_\beta + (\lfloor \ell_\beta/2 \rfloor -1)\log \bb E_x^{\beta, \ell_\beta} \rme^{-\gamma  (\sigma_{1} -\sigma_{0})} \Big),
\end{align*}
where in the last equality we used that the random variables $\sigma_{2k+1}-\sigma_{2k}$, $k =0, \dots, \lfloor \ell_\beta/2 \rfloor -1$, are i.i.d.\ and independent of the initial condition $x$ when $X$ is sampled according to $\bb P_x^{\beta, \ell_\beta}$. More precisely, from the strong Markov property and the symmetry of the potential $W$,   the law of $\sigma_{2k+1}-\sigma_{2k}$ is the same as the law of the hitting time $\tau$ of the point $-1+\rho_2$ under $\bb P_{-1+\rho_1}^{\beta, \ell_\beta}$.

To estimate $\bb E_{-1+\rho_1}^{\beta, \ell_\beta} (\rme^{-\gamma\tau}) $ we use that $f(x) := \bb E_x^{\beta, \ell_\beta} (\rme^{-\gamma  \tau} )$ can be characterized as the unique bounded solution to  
\[
\begin{cases} L_\beta  f = \gamma f  & \text{on }  (-\infty,   -1+\rho_2) \\ f(-1+\rho_2) = 1, \end{cases}
\]
where $L_\beta f =    \frac 1{2\beta}  f'' -  U'_\beta  f'$. Recalling~\eqref{He}, by ground state transformation, $ h= f \psi_\beta  $ satisfies
\[
\begin{cases}
H_\beta h =   (E_{0,\beta}- \gamma ) h &\text{on }  (-\infty,   -1+\rho_2) \\ h(-1+\rho_2) =  \psi_\beta(-1+\rho_2).
\end{cases}
\]
Choosing  $\gamma > \lim_{\beta} E_{0,\beta} = 1$, it follows that $h''\ge 2\beta (\gamma - E_{0,\beta}) > 0$ for any $\beta$ large enough.  Since $h$ is bounded, it is therefore a non decreasing function, whence $h(-1+\rho_1) \leq h(-1+\rho_2) = \psi_\beta(-1+\rho_2) $ for any $\beta$ large enough. We thus obtain $f(-1+\rho_1) \le \psi_\beta(-1+\rho_2)/\psi_\beta(-1+\rho_1)$. By Proposition~\ref{Prop_Semicl} (iii) there exists $C_\rho>0$ such that for any $\beta$ large enough
$ f(-1+\rho_1)   \le  \rme^{-\beta C_\rho} $. The statement follows.   
\end{proof}

We next analyze the transition probabilities of the chain $(Y_k)$ for $\beta\to \infty$. The transitions from $-1+\rho_2$ to $1-\rho_1$ and, symmetrically, $1-\rho_2$ to $-1+\rho_1$ correspond to atypical behavior of the underlying diffusion~\eqref{diff}.

\begin{lemma}
\label{lemmap}
\begin{equation}
\label{asymptp}      
\varlimsup_{\rho_2\to 0} \varlimsup_{\rho_1\to 0} \varlimsup_{\beta\to\infty} \frac 1\beta \log p_{\rho, \beta}  \le  - C_\mathrm{W}. 
\end{equation}
\end{lemma}

\begin{proof}
As follows by direct computation, the quasi-potential associated to the rate function in~\eqref{ratefuncI} is $2U$. In view of Proposition~\ref{LDPdiff} and since $C_\mathrm{W} = 2 U(0)$, the proof is achieved by the arguments in \cite[\S~4.2]{FW}.
\end{proof}

\begin{lemma}
\label{lemmanuovo}
Let $\mc N^\pm_{\rho, \beta} =  | \{  1\le k \le  \mc N_{\rho, \beta}  \colon  Y_{k-1} = \pm 1 \mp \rho_2  , \,  Y_{k} = \mp 1 \pm \rho_1 \} |$. Then, for any $ k\ge 0$, 
\[
\varlimsup_{\rho_2\to 0}  \varlimsup_{\rho_1\to 0}  \varlimsup_{\beta\to\infty}  \sup_{x\in \bb R} \frac 1\beta \log \bb P_x^{\beta, \ell_\beta} (\mc N^{-}_{\rho, \beta} + \mc N^+_{\rho, \beta} \ge k ) \le -k ( C_\mathrm{W} - \alpha).
\]
\end{lemma}

\begin{proof}
By  the graphical construction of the discrete time Markov chain $(Y_k)$, on the event $\{\mc N_{\rho, \beta} \le \ell_\beta\}$ the random variable $\mc N^{-}_{\rho, \beta} + \mc N^+_{\rho, \beta} $ is stochastically dominated by a random variable $Z$ having binomial distribution with parameters $(2\lfloor\ell_\beta\rfloor, p)$. By the exponential Chebyshev inequality,
\[   
\bb P(Z \ge k) \le \exp\Big\{ - 2\lfloor\ell_\beta\rfloor I_p \Big( \frac {k}{2\lfloor\ell_\beta\rfloor} \Big) \Big\}, \quad  I_p(z) =   z \log \frac zp + (1-z) \log \frac{1-z}{1-p}.
\]
By elementary computations we conclude by Lemmata~\ref{superexp} and~\ref{lemmap}.  
\end{proof}

In order to show exponential tightness, recalling the Fr\'echet-Kolmogorov compactness criterion, for $X\in L^p((0,1))$ we introduce 
\begin{equation}
\label{contmod}    
\omega_h(X) = \int_0^{1-h}\!\rmd s\, |X_{s+h} - X_{s}|^p + \int_0^h\!\rmd s\, |X_{s}|^p + \int_{1-h}^1\!\rmd s\,  |X_{s}|^p , \quad h\in (0,1). 
\end{equation}

\begin{lemma}
\label{PropTightness}
For each $L, \zeta>0$,
\begin{equation}
\label{contmod2}
\lim_{h \downarrow 0}  \varlimsup_{\beta\to \infty}   \sup_{|x|\le L}  \frac 1\beta \log  (\bb P_x^{\beta,\ell_\beta} \circ \imath^{-1}_{\eps_\beta} )  ( \omega_h > \zeta)  = - \infty. 
\end{equation}
\end{lemma}

\begin{proof}
The last two terms in the right-hand side of~\eqref{contmod} can be  controlled by Lemma~\ref{p:1}. By the same lemma, it suffices to control the first term in the right-hand side of~\eqref{contmod} for $p=2$. Given $\delta>0$, define $I_\delta^\pm := \{x\colon |x\mp 1|<\delta\}$ and introduce the sets
\begin{align*}
& D_{0,\rho}  =   \{s \in (0, 1-h) \colon  ||X_s| -1|  \ge \rho_2 \} \cap \{s \in (0, 1-h)  \colon  ||X_{s+h}| -1|  \ge\rho_1 \} ,\\ &  D_{\pm,\rho}    = \{s \in (0, 1-h)  \colon   X_s\in  I_{\rho_2}^\pm \} \cap \{s   \in (0, 1-h) \colon ||X_{s+h}| -1|  \ge \rho_1 \} ,
\\ & \hat D_{\pm,\rho}    = \{s \in (0, 1-h)  \colon X_{s+h} \in I_{\rho_1}^\pm\} \cap \{s \in (0, 1-h) \colon  ||X_{s}| -1| \ge \rho_2 \},
\\ & A_{\pm, \rho} = \{  s \in (0, 1-h) \colon X_s \in I_{\rho_2}^\pm,  X_{s+h}\in  I_{\rho_1}^\pm  \}, 
\\ &   C_{\pm, \rho}  = \{  s  \in (0, 1-h)  \colon X_s \in I_{\rho_2}^\pm , X_{s+h} \in I_{\rho_1}^\mp \}, 
\end{align*}
that form a partition of $(0,1-h)$. Letting
\[
c_{0,\rho} = \sup \Big\{\frac{|x|^2}{W(x)}\colon ||x|-1| \ge \rho_1 \Big\}, \quad c_{\pm,\rho} =  \sup \Big\{ \frac{|x- (\pm 1)|^2}{W(x)}\colon ||x|-1| \ge \rho_1 \Big\},     
\]
we have
\[
\int_{D_{0,\rho}}\!\rmd s\, |X_{s+h} - X_{s}|^2  \le   4 c_{0,\rho} \int_0^1\! \rmd s\, W(X_s),          
\]
while, for $D_\rho = D_{\pm,\rho} , \hat D_{\pm,\rho}$,
\[
\int_{D_{\rho}}\!\rmd s\, |X_{s+h} - X_{s}|^2 \le 2 \left(\rho_2^2 + c_{\pm,\rho} \int_0^1\!\rmd s\, W(X_s) \right).  
\]
Furthermore,
\[
\int_{A_{\pm,\rho}}\!\rmd s\,  |X_{s+h} - X_{s}|^2 \le (2\rho_2)^2.      
\]
Finally, recalling the stopping times $\sigma_k$ introduced in~\eqref{sigmak} and the $\mc N^\pm_{\rho,\beta}$'s defined in Lemma~\ref{lemmanuovo}, let $S^\pm$ be the ordered collection of $\sigma_{2k+1}$ such that $Y_{2k+1} = \pm 1 \mp \rho_2$ and $Y_{2k+2} = \mp 1 \pm \rho_1$. Set also $S^\pm =\{  \ell_\beta s_1^\pm , \dots,  \ell_\beta  s_{\mc N^\pm_{ \beta, \rho}}^\pm \}$.  By construction, 
\[
C_{ \pm,\rho} \subset \bigcup_{i=1}^{\mc N^\pm_{\rho, \beta}} [s_i^{\pm} -h , s_i^{\pm}],
\] 
so that
\[
\int_{C_{ \pm,\rho}}\!\rmd s\,  |X_{s+h} - X_{s}|^2 \le [2(1+ \rho_2)]^2  \mc N^\pm_{\rho, \beta}\, h.          
\]
Gathering the previous estimates and using Lemmata~\ref{lemma_Pgrande} (i) and~\ref{lemmanuovo} the statement
follows.
\end{proof}

\begin{proof}[Proof of Theorem~\ref{mainthmldp}: exponential
tightness] 
We show that the family of probability measures $\{\mu_{\beta}\}$ on $L^p((0,1))$ is exponentially tight, i.e., there exists a sequence of compacts $\mc K_j\subset L^p((0,1))$ such that
\[
\lim_{j\to\infty} \varlimsup_{\beta\to \infty}  \frac 1\beta
\log  \mu_\beta (\mc K_j^\mathrm{c}) =- \infty .     
\]
Recalling~\eqref{contmod} and the Fr\'echet-Kolmogorov compactness criterion, by a straightforward inclusion of events, see, e.g., \cite[\S~8]{Billingsley}, it suffices to show that for each $\zeta>0$
\[
\lim_{h \downarrow 0}  \varlimsup_{\beta} \frac 1\beta \log  \mu_\beta (\omega_h > \zeta)  = - \infty.
\]
This follows directly from Lemmata~\ref{t:5.1} and~\ref{PropTightness}.
\end{proof}

\subsection*{Upper bound}

We first prove the following local upper bound. 

\begin{lemma}
\label{t:5.7}
Fix $m\in {\rm BV}\big((0,1);\{-1,1\}\big)$ and sequences $m_k\to m$
in $L^p((0,1))$, $\zeta_k\to 0$. Then
\[
\varlimsup_{k\to \infty} \varlimsup_{\beta\to \infty}  \frac 1\beta
\log \mu_\beta  (\mc O^p_{\zeta_k}(m_k)) \le - (C_\mathrm{W}-\alpha) |S(m)|.     
\]
\end{lemma}

\begin{proof}
Recalling assumption~\eqref{ipell}, by  Lemma~\ref{lemma_staz0} it suffices to show the statement with $\mu_\beta$ replaced by $\bb P^{\beta,\ell_\beta}_{\pi_\beta}\circ \imath_\beta^{-1}$ and $\mc O^p_{\zeta_k}(m_k)$ replaced by 
\[
\Big\{ X \in L^p((0,1)) \colon  \int_{T_\beta/\ell_\beta}^{1- T_\beta/\ell_\beta} \! \rmd s \,   |X_s - m_k(s)|^p < (2\zeta_k)^p   \Big\},
\]
where $T_\beta$ satisfies $\lim_{\beta} \beta^{-1} T_\beta = \infty$ and $\lim_{\beta} \ell_\beta^{-1} T_\beta = 0$. By Proposition~\ref{Prop_Semicl} (iv) and  Lemma~\ref{p:1}, the previous statement is proven once we show that for each sequence $\zeta_k \to 0$ and $L>0$
\begin{equation}
\label{5.71}
\varlimsup_{k\to \infty} \varlimsup_{\beta\to \infty} \sup_{|x|\le L}\frac 1\beta \log\big( \bb P^{\beta,\ell_\beta}_{x}\circ \imath_\beta^{-1}\big) \big(\mc O^p_{\zeta_k}(m_k)\big) \le - (C_\mathrm{W}-\alpha) |S(m)|.
\end{equation}
To prove it, we first note that there exists a sequence $\zeta'_k \to 0$ such that the inclusion $\mc O_{\zeta_k}(m_k) \subset \mc O_{\zeta'_k}(m) $ holds for every $k$.  Moreover, recalling the definition of $\mc N^\pm_{\rho,\beta}$ in Lemma~\ref{lemmanuovo}, there exist sequences $0< \rho_{1,k}< \rho_{2,k}$ with $\rho_{2,k} \to 0$ such that
\[
\mc O_{\zeta'_k}(m) \subset \imath_{\beta}^{-1} \Big (\mc N^+_{\rho_k,\beta}  + \mc N^-_{\rho_k,\beta} \ge |S(m)| \Big), \quad k\in \bb N, 
\]
where $\rho_k = (\rho_{1,k}, \rho_{2,k})$. Therefore~\eqref{5.71} follows from Lemma~\ref{lemmanuovo}.
\end{proof}

\begin{proof}[Proof of Theorem~\ref{mainthmldp}: upper bound on compacts]
The proof is achieved by showing that the local upper bound in Lemma~\ref{t:5.7} implies the upper bound for compacts. The corresponding somewhat technical but general argument is based on a min-max lemma and detailed for completeness.  Let $\mathrm{BV}_{\le n} = \{ m \in \mathrm{BV} \big( (0,1);\{-1,1\} \big) : |S(m)| \le n \} $ so that $\mathrm{BV} \big( (0,1);\{-1,1\} \big) = \bigcup_{n} \, \mathrm{BV}_{\le n} $.  Since $ \mathrm{BV}_{\le n}$  is a compact subset of $L^p((0,1))$, for each $\zeta>0$ there exists $K\in \bb N$ and $m_1, \dots, m_K$ in $\mathrm{BV}_{\le n}$ such that $ \mathrm{BV}_{\le n} \subset \bigcup_i \mc O^p_\zeta (m_i) = : \mc A_{n, \zeta} $.

For $\zeta>0$ and $m \in \mathrm{BV}_{\le n}$ set
\[
j_\zeta(m) = - \varlimsup_\beta \frac 1\beta \log \mu_\beta(\mc O^p_\zeta (m)).  
\]   
Let also the function $J^1_{n, \zeta} \colon \mc A_{n, \zeta} \to [0, \infty)$ be defined by
\[
J^1_{n, \zeta}  (X) = \min_{i : X\in \mc O^p_\zeta(m_i)} j_\zeta(m_i).  
\]
Let $\{\sigma_k\}$ be as defined in~\eqref{sigmak} and $\mc N^\pm_{\rho,\beta}$ as in Lemma~\ref{lemmanuovo}. Introduce the
event $ \mc B_n = \imath_{\eps_\beta}^{-1} \Big (\mc N^+_{\beta, \rho} + \mc N^-_{\beta, \rho} \ge n \Big) $ and set
\[
c_n  =   -  \varlimsup_\beta \frac 1\beta \log \mu_\beta (  \mc B_n ) = 
-  \varlimsup_\beta \frac 1\beta \log \mu_{\beta, \ell_\beta}
(\mc N^+_{\beta, \rho}  + \mc N^-_{\beta, \rho} \ge n ) . 
\]
By the Urysohn lemma there exists a continuous function $J^2_{n, \zeta} \colon L^p((0,1)) \to [0, \infty)$ such that
\[
J^2_{n, \zeta}  (X) = \begin{cases} J^1_{n, \zeta}  (X) & \textrm{ if }  X\in \mc A_{n,\zeta},  \\ c_n   &  \textrm{ if }  X\in \mc A^\mathrm{c}_{n,2\zeta} .
\end{cases}
\]
Set finally $J_{n, \zeta} = J^2_{n, \zeta} \wedge c_n$.

We claim that for each $\zeta, n$ and each open $\mc G\subset L^p((0,1))$
\begin{equation} 
\label{claim1} 
\varlimsup_\beta \frac 1\beta \log \mu_\beta (\mc G) \le - \inf_{X\in \mc G } J_{n, \zeta}(X)
\end{equation}
and that, for each $X \in L^p((0,1))$,
\begin{equation} 
\label{claim2} 
\varliminf_{n}\varliminf_{\zeta\to 0} J_{n, \zeta} (X) \ge I (X) ,
\end{equation}
where $I$ is the rate function defined in~\eqref{Imod}. Since $X\mapsto J_{n, \zeta}(X)$ is continuous, by the min-max lemma in~\cite[App.~2, Lemma~3.2]{KiLa}, this claim implies the upper bound on compacts.

In order to prove~\eqref{claim1}, we introduce the sets
\begin{gather}
\mc W_\delta = \Big\{  X \in L^p((0,1)) \colon  \int_0^1 \rmd s \,W(X_s)) \le \delta \Big\}, \\  \mc X_{R} =  \Big\{ X \in L^p((0,1)) \colon  \sup_{0< t< \ell_\beta}  |X_t|  \le R \Big\},
\end{gather}
and observe that,  by Lemmata~~\ref{p:1},~\ref{lemma_Pgrande} (i), and~\ref{t:5.1}, for each $\delta, R>0$,
\begin{equation}
\label{Wdelta}
\lim_{\beta\to \infty}  \frac 1\beta  \log \mu_\beta (\mc
W_\delta^\mathrm{c}) =  - \infty ,     \quad      \lim_{\beta\to \infty}  \frac 1\beta  \log \mu_\beta (\mc  X_R^\mathrm{c}) =  - \infty    .
\end{equation}
We then write
\[
\mu_\beta (\mc G) \le \mu_\beta (\mc G \cap \mc B_n \cap \mc W_\delta \cap  \mc X_R) + \mu_\beta ( \mc W_\delta^\mathrm{c}) +    \mu_\beta ( \mc X_R^\mathrm{c}) + \mu_\beta ( \mc B_n^\mathrm{c} ).
\]
By~\eqref{Wdelta} and the above definitions of $j_\zeta$ and $c_n$, to complete the proof of~\eqref{claim1} it suffices to show that there exist $\delta= \delta(n, \zeta, R)>0$  such that  $ \mc B_n \cap \mc W_\delta \cap \mc X_R \subset \mc A_{n, \zeta}$. Given  $X\in C([0,1])$, satisfying $\mc N^+_{\rho,\beta}(\imath_{\epsilon_\beta}^{-1} X) + \mc N^+_{\rho,\beta} (\imath_{\epsilon_\beta}^{-1} X ) = k\le n$, we recursively define $(s_0,\alpha_0), \ldots, (s_k,\alpha_k)$ as follows. Let
$s_0 := \inf \{s\in [0,1] \colon |X_s| \ge 1- \rho_1 \} $ and set
$\alpha_0 = +$ if $X_{s_0} \ge 1-\rho_1$, $\alpha_0 = -$ if
$X_{s_0} \le -1+\rho_1$. Define iteratively
$\alpha_{j}= - \alpha_{j-1}$,
$s_j = \inf \{s\in [s_{j-1},1] \colon X_s = 1- \rho_1 \} $ if
$\alpha_{j-1} =+$ and
$s_j = \inf \{s\in [s_{j-1},1] \colon X_s = -1+ \rho_1 \} $ if
$\alpha_{j-1} =-$. Let now $C=C(\rho_1)<+\infty$ be such that
\[
|x+1|^2 \le C W(x)\ \mathrm{ for } \ x\le 1 -\rho_1 \quad \textrm{and} \quad |x-1|^2 \le C W(x)\ \mathrm{ for }\ x\ge -1 +\rho_1 , 
\]
that exists since $W$ has quadratic minima at $x=\pm 1$. The previous bound implies that if  $X \in \mc W_\delta \cap \mc X_R$ with $\delta$ small enough then $X\in \mc O^p_\zeta (m^{\alpha_0,k}_{s_1,\ldots, s_k})$.

To prove~\eqref{claim2} we first observe that, by Lemmata~\ref{lemmanuovo} and~\ref{t:5.1},  $c_n\to \infty$. Moreover Lemma~\ref{t:5.7} implies that $ \varliminf_{\zeta\to 0} j_\zeta (m) = I(m) $ uniformly for $m$ in compacts.  Since $\bigcup_{n} \bigcap_{\zeta} \mc A_{n, \zeta} = {\rm BV}((0,1);\{-1, 1\} )$ the claim follows. 
\end{proof}

\subsection*{Lower bound}

We start by proving the large deviation lower bound for the law of the diffusion~\eqref{diff}. Recall that $m^{\pm, 0} := \pm 1 $ and the definition of $m^{\pm,n}_{s_1, \dots,s_n}$ given in~\eqref{mn}.

\begin{lemma}
\label{LemmaLB}
Fix  $n\ge 0$, and, if $n\ge1$, also $0<s_1<\dots< s_n<1$. For any $\delta \in (0,1)$ 
and any open $L^p$-neighborhood $\mc O$ of  $m^{\pm,n}_{s_1, \dots,s_n}$,
\[
\varliminf_{\beta\to \infty} \inf_{x\in I_\delta^\pm} \frac 1\beta  \log (\bb P_x^{\beta,\ell_\beta} \circ  \imath^{-1}_{\beta} ) (\mc O) \ge - (C_\mathrm{W}-\alpha)n, \]
where $I^\pm_\delta := (\pm1-\delta, \pm 1+\delta)$. 
\end{lemma}

\begin{proof}
For the ease of presentation we detail first the cases $n=0,1$. For the case $n=0$, it is enough to show that  for each $\eta>0$, 
\begin{equation} 
\label{infs+}   
\lim_{\beta \to \infty} \bb P_x^{\beta,\ell_\beta} \Big(\frac{1}{\ell_\beta} \int_0^{\ell_\beta}\!\rmd t\, |X_t \mp 1 |^p  > \eta \Big) =  0  \quad  \text{ uniformly for } x \in I_\delta^\pm.
\end{equation}
Recalling~\eqref{ipell}, pick $x_0\in (0,1-\delta)$ such that  $2U(x_0) > \alpha$ and let  $\tau^{\pm} = \inf \{ t\geq 0\colon X_t  = \pm x_0 \} $ be the hitting time of $\pm x_0$. By  Proposition~\ref{LDPdiff} and the argument in \cite[Chap.~4, Theorem 4.2]{FW}  it follows that 
\begin{equation} 
\label{stimatau}     
\lim_{\beta \to \infty}   \bb P_x^{\beta,\ell_\beta} (\tau^\pm < \ell_\beta)  = 0 \quad  \text{ uniformly for }   x\in I_\delta^\pm.    
\end{equation}
Thus~\eqref{infs+} follows from  Lemma~\ref{lemma_Pgrande} (i).

For the case $n=1$, by symmetry, it is enough to consider $m^{-}_{s_1}$. 
Set $t_1 = \ell_\beta s_1$ and $\widehat X_t =  X_{(0\vee  t) \wedge \ell_\beta}$.
For $T, \eta >0$, recalling that $\theta$ denotes the time shift, we define
\begin{align*}  
\mc A^{-}  &  = \Big\{X \colon  \widehat{X}_{t_1-T}  \in I_\delta^- \text{ and } \frac{1}{\ell_\beta} \int_0^{t_1 -T}\! \rmd t\, | \widehat{X}_t+1|^p < \eta\Big\}, \\ \mc A^{+}   & =   \Big\{X \colon  \widehat{X}_{t_1+T}  \in I_\delta^+ \text{ and }  \frac{1}{\ell_\beta}  \int_{t_1 +T}^{\ell_\beta} \! \rmd t\,  | \widehat{X}_t- 1|^p < \eta\Big\},  \\ \mc A^{0}  &  =        \Big\{ X\colon  \widehat{X}_{t_1 \pm T}  \in I_\delta^\pm \text{ and }  | \widehat{X}_t|  < 1 + \delta  \text{ if } |t-t_1| <T    \Big\}.  
\end{align*}
We define, for $a\in \{0, -, +\}$ and $k\in \bb Z$, $\mc A_k^{a}  = \theta_{2k T} \mc A^a$. It is straightforward to verify that, given $\mc O$ as in the statement, there exist $\eta, \zeta>0$ such that, for $\beta$ large enough and $\bar k = \lfloor \ell_\beta \zeta / T \rfloor$,
\[       
\imath_{\beta}^{-1}( \mc O)  \supset  \bigcup_{k= - \bar k  }^{\bar k }  \mc C_k, \quad    \mc C_k = \mc A_k^{-} \cap  \mc A_k^{0} \cap \mc A_k^{+}.
\] 
By the Bonferroni inequality,
\[
(\bb P_x^{\beta,\ell_\beta} \circ  \imath^{-1}_{\beta} )(\mc O) \ge \bb P_x^{\beta,\ell_\beta} \Big(\bigcup_{k= - \bar k  }^{\bar k }  \mc C_k \Big) \ge  \sum_{k= - \bar k  }^{\bar k } \bb P_x^{\beta,\ell_\beta} (  \mc C_k )   -  \sum_{- \bar k \le h< k \le \bar k  }\bb P_x^{\beta,\ell_\beta} ( \mc C_h \cap \mc C_k  ). 
\]

We claim that  
\begin{equation} 
\label{infs}  
\varliminf_{T \to \infty} \varliminf_{\beta\to \infty} \inf_{|k|  \le \bar k } \inf_{x\in I_\delta^-} \frac 1\beta  \log  \bb P_x^{\beta,\ell_\beta} (  \mc C_k  )   \ge   -  C_\mathrm{W}   ,      
\end{equation}
and 
\begin{equation} 
\label{infs2}  
\varlimsup_{T \to \infty} \varlimsup_{\beta\to \infty} \sup_{- \bar k \le h< k \le \bar k } \sup_{x\in I_\delta^-} \frac 1\beta  \log  \bb P_x^{\beta,\ell_\beta} (  \mc C_h \cap \mc C_k )   \le   - 2  C_\mathrm{W},      
\end{equation}
which, recalling~\eqref{ipell}, readily imply the statement for $n=1$.

In order to prove~\eqref{infs} we first observe that~\eqref{infs+} implies
\[
\lim_{\beta\to \infty} \bb P_x^{\beta,\ell_\beta} (\theta_{t_1+ T} \mc A^+)   = 1 \quad \text{uniformly for } x\in I_{\delta}^+. 
\]
Moreover, as it follows from Lemma~\ref{Propsemigruppo} and again~\eqref{infs+},  
\[
\lim_{\beta\to \infty} \bb P_x^{\beta,\ell_\beta} (\mc A^-_k) = 1 \quad \text{uniformly for } x\in I_{\delta}^- \text{ and }  |k| \le  \bar k. 
\]
Finally, from Proposition~\ref{LDPdiff} and $C_\mathrm{W} = 2U(0)$, we deduce 
\begin{equation}
\label{infs++}
\varliminf_{T \to \infty} \varliminf_{\beta \to \infty}  \inf_{y\in I_\delta^-}    \frac  1\beta \log P_x^{\beta,\ell_\beta}  (\theta_{t_1-T}   \mc A^0  )
\ge  - C_\mathrm{W}. 
\end{equation}
By the Markov property and using the identities 
\[
\theta_{t_1 + (2k+1)T}  \id_{\mc A_k^+} =   \id_{\theta_{t_1+T} \mc A^+}, \qquad \theta_{t_1 + (2k-1)T} \id_{\mc A_k^0}  = \id_{\theta_{t_1-T} \mc A^0},
\]
we have
\[
\begin{split}
\bb P_x^{\beta,\ell_\beta} (\mc C_k)  &  = 
\bb P_x^{\beta,\ell_\beta}  (\mc A_k^ - \cap \mc A_k^0 \cap \mc A_k^+)  =  \bb E_x^{\beta,\ell_\beta} ( \id_{\mc A_k^-}  \id_{\mc A_k^0}   \, \bb  E_{ X_{t_1 + (2k+1)T}}^{\beta,\ell_\beta} (\id_{  \theta_{t_1+T} \mc A^+} )) \\  & \ge \inf_{y\in I_\delta^+} \bb P_y^{\beta,\ell_\beta}(    \theta_{t_1+T} \mc A^+  ) \, \bb  E_x^{\beta,\ell_\beta}   ( \id_{\mc A_k^-}  \id_{\mc A_k^0} ) \\ & = \inf_{y\in I_\delta^+}    \bb P_y^{\beta,\ell_\beta}(  \theta_{t_1+T} \mc A^+) \,  \bb  E_x^{\beta,\ell_\beta} (\id_{\mc A_k^-} \, \bb  E^{\beta,\ell_\beta}_{ X_{t_1 + (2k-1)T}} (   \id_{  \theta_{t_1-T}   \mc A^0})) \\  & \ge \inf_{y\in I_\delta^+}    \bb P_y^{\beta,\ell_\beta}(  \theta_{t_1+T}  \mc A^+) \inf_{z\in I_\delta^-}    \bb P_z^{\beta,\ell_\beta}(   \theta_{t_1-T} \mc A^0) \, \bb P_x^ \beta (A_k^-), 
\end{split}     
\]
and the bound~\eqref{infs} follows. 

To prove~\eqref{infs2}, again by Proposition~\ref{LDPdiff} and $C_\mathrm{W} = 2U(0)$, we deduce 
\begin{equation}
\label{infs+++}
\varlimsup_{T \to \infty} \varlimsup_{\beta \to \infty}  \sup_{y\in
I_\delta^-}    \frac  1\beta \log \bb P_x^{\beta,\ell_\beta}  (\theta_{t_1-T}   \mc A^0  ) \le  - C_\mathrm{W}. 
\end{equation}
By the Markov property, for $h<k$,
\[
\begin{split}
\bb P_x^{\beta,\ell_\beta} (\mc C_h \cap \mc C_k)  &  \le  
\bb P_x^{\beta,\ell_\beta}  (\mc A_h^0\cap  \mc A_k^0 )  =
\bb E_x^{\beta,\ell_\beta} ( \id_{\mc A_h^0}   \, \bb  E_{ X_{t_1 + (2k-1)T}}^{\beta,\ell_\beta} (    \id_{  \theta_{t_1-T} \mc A^0} )) \\  & \le \sup_{y\in I_\delta^-}     \bb P_y^{\beta,\ell_\beta}(     \id_{  \theta_{t_1-T} \mc A^0}  ) \, 
\bb  E_x^{\beta,\ell_\beta}    ( \bb  E_{ X_{t_1 + (2h-1)T}}^{\beta,\ell_\beta} (    \id_{  \theta_{t_1-T} \mc A^0} )) \\
& \le    \sup_{y\in I_\delta^-}   \big(   \bb P_y^{\beta,\ell_\beta}(     \id_{  \theta_{t_1-T} \mc A^0}  )      \big)^2,
\end{split}     
\]
and the bound~\eqref{infs2} follows by~\eqref{infs+++}.

For the general case, by symmetry, it is enough to consider  $m^{-, n}_{s_1, \dots, s_n}$. Set $t_i = \ell_\beta s_i$, $i=1, \dots, n$, and for $\eta, T>0$ introduce the events $\mc A^{-}$ as before and
\begin{align*}     
\mc A^{+}   & = \Big\{ X \colon \widehat{X}_{t_n+T}  \in I_\delta^+ \text{ and }   \frac{1}{\ell_\beta}  \int_{t_n +T}^{\ell_\beta}\! \rmd t\,  | \widehat{X}_t- 1|^p   < \eta   \Big\}, \\ \mc A^{0,i}  &  =        \Big\{ X \colon  \widehat{X}_{t_i \pm T}  \in I_\delta^\pm \text{ and }  | \widehat{X}_t|  < 1 + \delta  \text{ if } |t-t_i| <T  \Big\}. 
\end{align*}
As before, for $a\in \{\{0,i\}, -, +\}$ and $k\in \bb Z$ set $\mc A_k^{a}  = \theta_{2kT} \mc A^a$, and for $k,h\in \bb Z$, we let
\[
\mc B_{k,h}^{i, i+1} = \Big\{X \colon  X_{t^+_{i,k}}, X_{t^-_{i+1,h}}  \in I_\delta^i, \text{ and }  \frac{1}{\ell_\beta} \int_{t^+_{i, k}}^{t^-_{i+1,h}}\! \rmd t\, | X_t- (-1)^{i+1}|^p   < \eta  \Big\},
\]
where $t^\pm_{i, k} =   t_i + (2k\pm1)T$ and $I_\delta^i=I_\delta^+$, respectively $I_\delta^i=I_\delta^-$, if $i$ is odd, respectively even. Again, it is straightforward to verify that, given $\mc O$ as in the statement, there exist $\eta, \zeta>0$ such that,  for $\beta$ large enough and $\bar k = \lfloor  \ell_\beta \zeta/T   \rfloor $,
\[
\imath_{\beta}^{-1}( \mc O)     \supset   \bigcup_{k_1= - \bar k  }^{\bar k } \dots   \bigcup_{k_n= - \bar k  }^{\bar k }  \mc  C_{k_1, \dots, k_n},  
\]
with
\[
\mc   C_{k_1, \dots, k_n} = \mc A_{k_1}^{-} \cap  \mc A_{k_1}^{0,1} \cap    \mc B_{k_1, k_2}^{1,2} \cap  \mc A_{k_2}^{0,2}   \cap \dots \cap \mc A_{k_{n-1}}^{0,n-1} \cap \mc B_{k_{n-1}, k_n}^{n-1,n}   \cap   \mc A_{k_n}^{0,n} \cap \mc A_{k_n}^{+}.
\] 
From~\eqref{infs+}, ~\eqref{infs++} and~\eqref{stimatau}, by using the Markov property as in the case $n=1$, we get
\[ 
\varliminf_{T \to \infty} \varliminf_{\beta\to \infty} \frac 1\beta  \log  \bb P_x^{\beta,\ell_\beta} (   \mc   C_{k_1, \dots, k_n}) \ge - n C_\mathrm{W},
\]
uniformly for  $|k_i| \le \bar k $ and  $x\in I_\delta^-$. Again by the Markov property and by~\eqref{infs+++},
\[
\varlimsup_{T \to \infty} \varlimsup_{\beta\to \infty} \frac 1\beta  \log  \bb P_x^{\beta,\ell_\beta} (\mc C_{h_1, \dots, h_n}  \cap  \mc   C_{k_1, \dots, k_n}) \le - (n +  |\{ i\colon k_i \ne h_i\}|) C_\mathrm{W},
\]
uniformly for  $|h_i|,|k_i| \le \bar k $, $ (h_1, \dots, h_n)\neq (k_1, \dots, k_n) $ and  $x\in I_\delta^-$. By using the Bonferroni inequality as before the statement follows. 
\end{proof}

\begin{proof}[Proof of Theorem~\ref{mainthmldp}: lower bound]
Fix $n\ge 0$  and, if $n\ge1$, also $0<s_1<\dots< s_n<1$.
It is enough to show that for any open $L^p$-neighborhood $\mc O$ of  $m^{\pm,n}_{s_1, \dots,s_n}$,
\[
\varliminf_{\beta\to \infty} \frac 1\beta\log\mu_\beta(\mc O) \ge - (C_\mathrm{W}-\alpha)n.
\]
By symmetry, it is enough to consider $m^{-, n}_{s_1, \dots, s_n}$. By the representation~\eqref{rme} and~\eqref{repwp}, 
\begin{align*} 
\mu_\beta(\mc O)  & =  \int_{\bb R^2}\! \wp_\beta (\rmd x\,\rmd y) (\bb P_{x,y }^{\beta,\ell_\beta} \circ\imath^{-1}_{\beta}) (\mc O) \\ & =  \frac{1}{Z_\beta}  \int_\bb R\! \rmd x \int_\bb R \! \bb P_x^{\beta,\ell_\beta} ( X_{\ell_\beta} \in \rmd y)  \frac{\psi_{0, \beta}(x)}{\psi_{0, \beta}(y)} (\bb P_{x,y}^{\beta,\ell_\beta} \circ  \imath^{-1}_{\beta}) (\mc O)  \\    & \ge \frac{1}{Z_\beta} \inf_{y \in \bb  R}  \frac{1}{\psi_{0, \beta}(y)} \int_\bb R\! \rmd x\, \psi_{0, \beta}(x) (\bb P_x^{\beta,\ell_\beta} \circ \imath^{-1}_{\beta}) (\mc O) \\ & \ge \frac{1}{Z_\beta}  \inf_{y \in \bb  R}  \frac{1}{\psi_{0, \beta}(y)}  \inf_{z\in I_\delta^-} (\bb P_z^{\beta,\ell_\beta} \circ  \imath^{-1}_{\beta}) (\mc O) \int_{I_\delta^-}\! \rmd x \,  \psi_{0, \beta}(x). 
\end{align*} 
By items (iii) and (iv) in Proposition~\ref{Prop_Semicl}, 
\[
\varliminf_{\beta\to\infty} \frac 1\beta \inf_{y \in \bb  R} \log \frac{1}{\psi_{0, \beta}(y)} \ge 0,
\]
and the proof is concluded by Corollary~\ref{t:3.2},~\eqref{stimaZpw}, and Lemma~\ref{LemmaLB}.
\end{proof}

\section{Sharp asymptotics}
\label{sec:6}

The aim of this section is to prove the sharp asymptotics stated in Theorem~\ref{theorem_sharp}. Accordingly, we fix $p\in [1, \infty),n\ge 0$,  and the sequences $(\rho_{k,\beta})_{\beta>0}$, $k=1,2$, meeting  the conditions in the statement.

We choose  a sequence $(T_\beta)_{\beta>0} $ such that
\begin{equation} 
\label{cond_T}
\lim_{\beta\to \infty} \frac{T_\beta}{\beta^{3/2}} = \infty,  \quad   \lim_{\beta\to \infty}  \frac{T_\beta }{\ell_\beta} =  0, \quad    \lim_{\beta\to \infty}  \frac{T_\beta }{\ell_\beta \rho^p_{1, \beta}} = \infty, 
\end{equation}
which is possible in view of assumption~\eqref{ipells}. We recursively define the sequence of stopping times $(\tau_k)$ on $C([0, \ell_\beta])$, by setting $\tau_0 = 0$ and, for $k\in \bb N$,
\[    
\tau_k    =     \inf\{ t\in [T_\beta+\tau_{k-1}, \ell_\beta - 2T_\beta) :    X_{t} = 0 \} \wedge   (\ell_\beta  - 2T_\beta ). 
\]
Further we set  
\[
\mc Z  =  \left|  \{ k \in \bb N :   \tau_k < \ell_\beta - 2T_\beta\}\right|. 
\]
Observe that $\tau_k - \tau_{k-1} \ge T_\beta $ for $k \le \mc Z$. Let also 
\[
\mc N =  \left|  \{ k \in \bb N :   X_{T_\beta + \tau_{k-1}}  X_{T_\beta + \tau_{k}}  <0 \}\right|,  
\]
\[
\mc B_{n,\delta}   =   \bigcap_{k=1}^n  \{   \tau_k -\tau_{k-1} >  \ell_\beta \delta \} \cap \{\tau_n < \ell_\beta (1- \delta)\}.
\]

\begin{proposition}
\label{prop_sharp_2}
As $\beta \to \infty$,
\begin{itemize}
\item[(i)]  $\displaystyle  \mu_{\beta, \ell_\beta}  ( \mc Z \ge n ) \sim   \frac{2^n}{n!} \Big(\frac{\ell_\beta}{\bar \ell_\beta} \Big)^n$.
\item[(ii)]  $\displaystyle \mu_{\beta, \ell_\beta}  ( \mc N \ge n )   \sim   \frac{1}{n!} \Big(\frac{\ell_\beta}{\bar \ell_\beta} \Big)^n$. 
\item[(iii)]  Let $(\delta_\beta)_{\beta>0}$ be a sequence such that  $\lim_{\beta}  \delta_\beta \ell_\beta /T_\beta = 0$. 
Then
\[
\mu_{\beta, \ell_\beta}  ( \mc N = n, \, \mc B_{n, \delta_\beta} )   \sim 
\frac{1}{n!} \Big(\frac{\ell_\beta}{\bar \ell_\beta} \Big)^n.
\]
\end{itemize}
\end{proposition}

Postponing the proof of this proposition, we first conclude the proof of Theorem~\ref{theorem_sharp}. To  this  end, we introduce the events
\[
\mc W_{\gamma}  = \Big\{ X: \frac{1}{\ell_\beta} \int_0^{\ell_\beta} \! \rmd t \, W(X_t) \le \frac {1+\gamma}{\beta} \Big\},\qquad \mc X_{R} =  \Big\{ X :  \sup_{0< t< \ell_\beta}  |X_t|  \le R \Big\}.
\]

\begin{lemma} 
\label{lemma_inclusione}
Fix $\gamma, R>0$, $z \in \bb N$ and $i=1,2$. Then, eventually as $\beta\to \infty$, 
\[
\left[ \imath_\beta \mc O_{\rho_{i,\beta}} (\mc M_{n-1}) \right]^{\rm c}  \cap \mc W_\gamma \cap \mc X_R \cap \{\mc Z\le z\}    \subset   \{ \mc  N  \ge n \}. 
\]
\end{lemma}

\begin{proof} Let $X \in\left[ \imath_\beta \mc O_{\rho_{i,\beta}} (\mc M_{0}) \right]^{\rm c}  \cap \mc W_\gamma \cap \mc X_R \cap  \{\mc Z\le z\}$. Set 
\[ 
J  = \{ j \in \{ 1, \dots, \mc Z(X)\} \colon   X_{T_\beta + \tau_{j-1}}   X_{T_\beta + \tau_{j}}  <0 \},  \quad  \bar n = |J|,     
\]
and, in the case $J \ne \emptyset$, let $j_1, \dots, j_{\bar n}$ be an increasing enumeration of the elements of $J$. Assuming with no loss of generality that $X_{\tau_{j_1}+T_\beta} >0$, let $m_J^{-} =  m^{-, \bar n}_{s_1, \dots, s_{\bar n}}$, with $s_k =   \tau_{j_k}/\ell_\beta$, $k=1, \dots, \bar n$, if $\bar n \ge 1$ and  $m_{\emptyset}^{-} = -1$. Since $X\in \mc X_R$, 
\begin{gather}
\label{stimarho}
\| \imath_\beta X- m^{-}_J\|_{L^p}^p \le  \frac{(3+ \mc Z(X)) T_\beta}{\ell_\beta} (1+R)^p   +    B_1(X) + B_2(X) ,
\end{gather}
where
\begin{gather*}
B_1(X) =  \frac{1} {\ell_\beta} \sum_{k=1}^{\bar n +1 } \sum_{m=0}^{j_k-j_{k-1} -1}  \int_{T_\beta+\tau_{j_{k-1}+m} }^{\tau_{j_{k-1} +1+m} }\! \rmd t \, |X_t- (-1)^k|^p,   \\ B_2(X) =   \frac{1}{\ell_\beta} 
\int_{T_\beta+\tau_{\mc Z(X)}}^{\ell_\beta- 2T_\beta}\! \rmd t \, |X_t- (-1)^{\bar n+1}|^p  \,  \id_{\tau_{\mc Z(X)} +T_\beta < \ell_\beta - 2T_\beta}.
\end{gather*}
From the assumptions on $T_\beta$ and $\mc Z(X) \le z$ it follows that the first term on the right-hand side of~\eqref{stimarho} satisfies $\lim_\beta  \beta \frac{(3+ \mc Z(X)) T_\beta}{\ell_\beta} (1+R)^p  = 0 $. To estimate $B_1(X)$ we consider first the case  $p\in [2, \infty)$. We observe that  $X \ge 0$ (resp.~$X\le 0$) on  $[T_\beta+\tau_{j_{k-1}+m} ,\tau_{j_{k-1} +1+m}]$ for any $k$ even (resp.~$k$ odd) and any  $m$. Therefore,  by using that  $(x+1)^2 \le C W(x)$ for $x<0$ and $(x-1)^2 \le C W(x)$ for $x>0$,  we deduce 
\begin{gather*}
B_1(X)   \le C(1+R)^{p-2} \frac{1+\gamma}{\beta}       .
\end{gather*}
The term $B_2(X)$ satisfies the same bound, as can be shown with the same reasoning using that $X\ge 0 $ (resp.~$X\le 0$) on $[T_\beta + \tau_{\mc Z}, \ell_\beta - 2T_\beta]$ if $\bar n$ is odd (resp.~even). Hence, eventually as $\beta \to \infty$, 
\begin{equation}
\label{pincopallo}
\| \imath_\beta X- m^{-}_{J}\|_{L^p}     \le   C \beta^{-1/p}      .
\end{equation}
On the other hand, if  $\bar n < n$, then $m^{-}_{J} \in \mc O_{\rho_\beta} (\mc M_{n-1})$ and thus $ \| \imath_\beta X- m^{-}_{J}\|_{L^p} \ge \rho_{i,\beta} $, which, together with~\eqref{pincopallo}, contradicts the  assumption on $\rho_{i,\beta}$. Therefore $\bar n\ge n$, which concludes the proof in the case  $p\in [2, \infty)$. For $p\in [1, 2)$ we use $\|\imath_\beta X -  m^{-}_{J} \|_{L^p}  \le  \|\imath_\beta X -  m^{-}_{J} \|_{L^2} \le C \beta^{-1/2}$, where we used~\eqref{pincopallo} with $p=2$ in the second inequality. 
\end{proof}

\begin{lemma}
\label{lemma_inclusione2}
Fix $\gamma, R>0$ and $ z \in \bb N$. Let also  $(\delta_\beta)_{\beta>0}$ be a positive sequence such that $\lim_\beta \delta_\beta = 0$ and  $\lim_\beta \delta_\beta\rho_{2,\beta}^{-p} = \infty$. Then, eventually as $\beta\to \infty$, 
\[
\{ \mc  N  = n  \}   \cap    \mc B_{n,\delta_\beta}  \cap \mc W_\gamma \cap \mc X_R   \cap \{\mc Z \le z\} \subset  \left[ \imath_\beta \mc O_{\rho_{2,\beta}} (\mc M_{n-1}) \right]^{\rm c}. 
\]
\end{lemma}

\begin{proof}
The statement is a direct consequence of the following two steps.

\smallskip
\noindent
{\it Step 1.} If $X\in  \{ \mc  N  = n  \}   \cap    \mc B_{n,\delta_\beta}  \cap 
\mc W_\gamma \cap \mc X_R \cap \{\mc Z \le z\} $ then there are $0= s_0 <s_1< \dots < s_n<1$ satisfying 
$s_k-s_{k-1}\ge \delta_\beta$, $k=1, \dots, n$,  $s_n\le 1-\delta_\beta$ and $*\in \{-,+\}$ such that
\[   \| \imath_\beta X - m^{*, n}_{s_1, \dots, s_n}\|_{L^p}  < \rho_{2,\beta}  . \]

The proof of this step is achieved by arguing as in the proof of Lemma~\ref{lemma_inclusione}.

\smallskip
\noindent
{\it Step 2.}  Let $0=s_0<s_1<\dots < s_n<1$ be such that 
$s_k-s_{k-1}\ge \delta_\beta$, $k=1, \dots, n$,  $s_n\le 1-\delta_\beta$ and $*\in \{-,+\}$. Then
\[  {\rm dist}_{L^p} (m^{*, n}_{s_1, \dots, s_n}, \mc M_{n-1} ) \ge   2 (\delta_\beta/2)^{1/p}   .   \]

Consider the intervals $I_k =(s_k-\delta_\beta/2, s_k+\delta_\beta/2)$, $k=1,\dots, n$, which are pairwise disjoint. For any $m\in \mc M_{n-1}$ there exist an interval $I_{\bar k}$ on which 
$m$ is constant. Hence, 
\[
{\rm dist}_{L^p} (m^{*, n}_{s_1, \dots, s_n}, \mc M_{n-1} )^p \ge     2^p \delta_\beta/2.
\]
\end{proof}

\begin{proof}[Proof of Theorem~\ref{theorem_sharp}]
For each $\gamma, R,z >0$ and $k=1,2$, by Lemma~\ref{lemma_inclusione}, eventually as $\beta \to \infty$, we have 
\[
\mu_\beta( [\mc O_{\rho_{k, \beta}}  (\mc M_{n-1})]^{\mathrm c}) \le   \mu_{\beta} (\mc N \ge n)  -  \mu_{\beta} (\mc W^{\mathrm c}_{\gamma}) -  \mu_{\beta} (\mc X^{\mathrm c}_{R}) -   \mu_{\beta} (\mc Z >  z).
\]
Proposition~\ref{prop_sharp_2} (ii) provides the sharp asymptotics of 
$\mu_{\beta} (\mc N \ge n) $. Choosing $\gamma>0$, $R$ large enough, and $z =n$, Lemmata~\ref{lemma_Pgrande} (ii),~\ref{p:1},~\ref{t:5.1} and Proposition~\ref{prop_sharp_2} (i) imply that the remaining terms are negligible with respect to $\mu_{\beta} (\mc N \ge n) $. Hence
\begin{equation}
\label{UB_sharp}
\mu_\beta( [\mc O_{\rho_{k, \beta}}  (\mc M_{n-1})]^{\mathrm c})  \le 
\frac{1}{n!} \Big(\frac{\ell_\beta}{\bar \ell_\beta} \Big)^n (1 + o_\beta(1)).
\end{equation}
For $k=1$ this proves the upper bound in the statement. To deduce the lower bound we write
\begin{align}
\label{prob_totali}
& \mu_\beta( \mc O_{\rho_{2,\beta}} (\mc M_{n})  \setminus \mc O_{\rho_{1, \beta}}  (\mc M_{n-1}))  \nonumber  \\ & \qquad = \mu_\beta( [\mc O_{\rho_{1, \beta}}  (\mc M_{n-1})]^{\mathrm c})  -\mu_\beta( [\mc O_{\rho_{1, \beta}}  (\mc M_{n-1})]^{\mathrm c}  \cap  (\mc O_{\rho_{2,\beta}} (\mc M_{n}))^{\mathrm c} )   \nonumber  \\ & \qquad \ge   \mu_\beta( [\mc O_{\rho_{1, \beta}}  (\mc M_{n-1})]^{\mathrm c})  -\mu_\beta( (\mc O_{\rho_{2,\beta}} (\mc M_{n}))^{\mathrm c} ).
\end{align}
In view of~\eqref{cond_T}, we can choose a sequence $(\delta_\beta)_{\beta>0}$ such that $\lim_\beta \delta_\beta\rho_{1,\beta}^{-p} = \infty$ and $\lim_{\beta}  \delta_\beta \ell_\beta /T_\beta = 0$. For each $\gamma, R,z >0$, by Lemma~\ref{lemma_inclusione2}, eventually as $\beta \to \infty$, 
we have 
\[ 
\mu_\beta( [\mc O_{\rho_{1, \beta}}  (\mc M_{n-1})]^{\mathrm c})  \ge   \mu_{\beta} (\mc N = n, \mc B_{n, \delta_\beta})  -  \mu_{\beta} (\mc W^{\mathrm c}_{\gamma}) -  \mu_{\beta} (\mc X^{\mathrm c}_{R}) -   \mu_{\beta} (\mc Z >  z).
\]
Proposition~\ref{prop_sharp_2} (iii) provides the sharp asymptotics of $\mu_{\beta} (\mc N = n, \mc B_{n, \delta_\beta}) $. Choosing $\gamma>0$, $R$ large enough, and $z =n$, Lemmata~\ref{lemma_Pgrande} (ii),~\ref{p:1},~\ref{t:5.1} and Proposition~\ref{prop_sharp_2} (i)  imply that the remaining terms are negligible with respect to $\mu_{\beta} (\mc N = n, \mc B_{n, \delta_\beta}) $. Hence,
\[  
\mu_\beta( [\mc O_{\rho_{1, \beta}}  (\mc M_{n-1})]^{\mathrm c})  \ge  \frac{1}{n!} \Big(\frac{\ell_\beta}{\bar \ell_\beta} \Big)^n (1 + o_\beta(1)).
\]
Recalling~\eqref{prob_totali} and using~\eqref{UB_sharp} with $k=2$ and $n$ in place of $n-1$, the lower bound of the statement follows. 
\end{proof}

\begin{proof}[Proof of Proposition~\ref{prop_sharp_2}.]
$~$

\smallskip
\noindent
{\it Proof of (i).}  Since $ \{\mc Z \ge n \} \in  \sigma(\{ X_t\}_{t\in [T_\beta, \ell_\beta - T_\beta]}) $, by Lemma~\ref{lemma_staz0}  it is enough to prove the statement for  $\bb P^\beta_{\pi_\beta}$ instead  of $\mu_{\beta, \ell_\beta}$. This will be achieved by induction. More precisely, set $T_{0, \beta} = T_\beta$ and fix sequences  $T_{1,\beta},\dots,  T_{n,\beta}$ such that $T_{k, \beta} / T_{k-1, \beta} \to \infty$, $k= 1, \dots,  n$, and $T_{n, \beta}/\ell_\beta \to 0$. We recursively define the sequence of stopping times $\sigma_0,\dots, \sigma_n $ on $C([0, \infty))$, by setting $\sigma_0 = 0$ and $  \sigma_k    =     \inf\{ t \ge T_\beta+\sigma_{k-1}:    X_{t} = 0 \} ,  \ k=1, \dots, n $. We claim that, for $k= 1, \dots,  n$, 
\begin{equation}
\label{induction_sigma}
\bb P^\beta_{\pi_\beta}  ( \sigma_k <  t    )   \sim  
\frac  {(t A_\mathrm{W} \sqrt \beta )^k}{  k! }   \rme^{- k \beta C_\mathrm{W}}      \   \   \   
\text{ uniformly for } t\in [T_{k,\beta}, \ell_\beta  ]   .  
\end{equation}
Since $T_{1, \beta}/T_\beta \to \infty$, the case $k=1$ follows from Theorem~\ref{t:sht}. To prove the inductive step, by the strong Markov property, 
\begin{gather*}
\bb P^\beta_{\pi_\beta}  ( \sigma_{k+1} <  t   )  =  \int_{T_\beta}^{t - T_\beta}  \!  \bb P^\beta_{\pi_\beta}  (   \sigma_k \in \rmd s  )  \bb P^\beta_{0} (\sigma_1 < t-s) . 
\end{gather*}
By Lemma~\ref{lemma_staz0} and the recursive hypothesis, it is enough to show 
\begin{gather*}   
\int_{T_\beta}^{t - T_\beta}\!  \bb P^\beta_{\pi_\beta}  (\sigma_k \in \rmd s  )  \bb P^\beta_{\pi_\beta}  (\tau < t-s -T_\beta)   \sim    \frac  {(t A_\mathrm{W} \sqrt \beta )^{k+1}}{  (k+1)! }   \rme^{- (k+1) \beta C_\mathrm{W}},   
\end{gather*}
uniformly for  $t\in [T_{k+1,\beta}, \ell_\beta] $. We first observe that
\begin{gather*}   
\int_{t- 2T_\beta}^{t - T_\beta}\!  \bb P^\beta_{\pi_\beta}  (   \sigma_k \in \rmd s) \bb P^\beta_{\pi_\beta}  (\tau < t-s -T_\beta) \le  \bb P^\beta_{\pi_\beta} ( \sigma_k< t-T_\beta)  \bb P^\beta_{\pi_\beta}   ( \tau\le T_\beta)  \\ \sim   \frac  { T_\beta  t^k   ( A_\mathrm{W} \sqrt \beta )^{k+1}}{ k! }   \rme^{- (k+1) \beta C_\mathrm{W}},
\end{gather*}
where we used Theorem~\ref{t:sht} and the recursive hypothesis. On the other hand, by  the same theorem, 
\begin{gather*}
\int_{T_\beta}^{t - 2 T_\beta} \! \bb P^\beta_{\pi_\beta}  (   \sigma_k \in \rmd s ) \bb P^\beta_{\pi_\beta}  (\tau < t-s -T_\beta) \\ =    A_\mathrm{W} \sqrt \beta    \rme^{-  \beta C_\mathrm{W}}   \int_{T_\beta}^{t - 2 T_\beta} \! \bb P^\beta_{\pi_\beta}  (   \sigma_k \in \rmd s )  (t-s-T_\beta)   ( 1   +   R_{\beta}(t-s)) ,
\end{gather*}
with  
\[    
\lim_{\beta \to \infty} \sup_{t\in [3T_\beta, \ell_\beta] }  \sup_{s'\in[ 2 T_\beta, t  - T_\beta ]} | R_\beta (s' ) |  =  0.
\]
By integration by parts, since   $ \bb P^\beta_{\pi_\beta}  (   \sigma_k  \le T_\beta  ) = 0$, 
\begin{gather*}
\int_{T_\beta}^{t - 2 T_\beta}  \!  \bb P^\beta_{\pi_\beta}  (   \sigma_k \in \rmd s ) (t-s-T_\beta)     =  \\  \int_{T_\beta}^{ T_{k,\beta}} \! \rmd s \,     \bb P^\beta_{\pi_\beta}  (   \sigma_k  \le s  ) +  \int_{T_{k,\beta}}^{t - 2 T_\beta} \! \rmd s \,  \bb P^\beta_{\pi_\beta}  (   \sigma_k  \le s  )    +   \bb P^\beta_{\pi_\beta}  (   \sigma_k  \le t- 2T_\beta  ) T_\beta    .
\end{gather*}
By the recursive  hypothesis and the choice of $T_{k, \beta}$,  
\[ 
\int_{T_{k,\beta}}^{t - 2 T_\beta} \! \rmd s\, \bb P^\beta_{\pi_\beta}  (   \sigma_k  \le s )   \sim \frac{ (A_\mathrm{W} \sqrt\beta)^k  \rme^{-  k \beta C_\mathrm{W}}      }{(k+1)!}    \left[ (t - 2 T_{\beta})^{k+1}     -  T_{k,\beta}^{k+1} \right]   , 
\]
while
\[
\int_{T_\beta}^{ T_{k,\beta}} \! \rmd s \,  \bb P^\beta_{\pi_\beta}  (   \sigma_k  \le s )    \le  T_{k, \beta} \bb P^\beta_{\pi_\beta}  (   \sigma_k  \le T_{k, \beta} ) \sim    \frac{ (A_\mathrm{W} \sqrt\beta)^k  \rme^{-  k \beta C_\mathrm{W}}}{ k!} T_{k, \beta}^{k+1}  ,     
\]
and
\[  
\bb P^\beta_{\pi_\beta}  (   \sigma_k  \le t- 2T_\beta  ) T_\beta      \sim 
T_\beta  (t-  2T_\beta)^k   \frac{ (A_\mathrm{W} \sqrt\beta)^k  \rme^{-  k \beta C_\mathrm{W}} }{ k!}.
\]
Gathering the previous bounds,~\eqref{induction_sigma} follows. 

\smallskip
\noindent
{\it Proof of (ii).} 
By item (i) it is enough to show that 
\[  
\mu_{\beta, \ell_\beta}  ( \mc N  =    n, \mc Z = n )   \sim   \frac{1}{n!} \Big(\frac{\ell_\beta}{\bar \ell_\beta}\Big)^n .
\]
Given  $ t_1 < t_2 $ and $x,y \in \bb R \cup \{{\rm f}\}$ denote by $\mu_{\beta, [t_1, t_2]}^{x,y}$ the $\phi^4_1$ measure on the interval $[t_1, t_2]$ with Dirichlet boundary conditions $x,y$ at the endpoints. When $x$ or $y$ equals ${\rm f}$ we understand free boundary conditions. In particular  $\mu_{\beta, [0, \ell]}^{{\rm f}, {\rm f}}  = \mu_{\beta, \ell}$. By the strong Markov property of $\mu_{\beta, \ell}$
\begin{gather*}   
\mu_{\beta, \ell} ( \cdot  | \mc Z = n)   \\ =    \int\!
\nu_{\beta, \ell}^n (\rmd t_1,\dots, \rmd t_n ) \,
\mu_{\beta, [0, t_1]}^{{\rm f},0}   \otimes  \mu_{\beta, [t_{1}, t_{2}]}^{0,0} \otimes  \cdots \otimes  \mu_{\beta, [t_{n-1}, t_{n}]}^{0,0} 
\otimes     \mu_{\beta, [t_n, \ell]}^{0, {\rm f} }   ,   
\end{gather*}
where
\[
\nu_{\beta, \ell}^n (\rmd t_1,\dots, \rmd t_n )   =   \mu_{\beta, \ell} (\tau_1\in \rmd t_1, \dots,\tau_n\in \rmd t_n | \mc Z = n ) .  
\]
Since the probabilities $\mu^{0,0}_{\beta, [t_k,  t_{k+1}]}$, $k=1, \dots, n-1$, $\mu^{{\rm f},0}_{\beta, [0,  t_{1}]}$, and  $\mu^{0, {\rm f}}_{\beta, [t_n,  \ell]}$  are invariant with respect to $X \mapsto -X$ we deduce  $\mu_{\beta, \ell_\beta} (\mc N = n, \mc Z = n) =  2^{-n} \mu_{\beta, \ell_\beta} (\mc Z = n)   $. The statement now follows from (i).

\smallskip
\noindent
{\it Proof of (iii).}  By the symmetry argument used in (ii), it is enough to show 
\[
\mu_{\beta, \ell_\beta}  ( \{\mc Z = n\} \cap \mc B_{n, \delta_\beta} )   \sim   \frac{2^n}{n!} \Big(\frac{\ell_\beta}{\bar \ell_\beta}\Big)^n.
\]
This is achieved by the induction argument used in (i), noticing that  Theorem~\ref{t:sht} implies
\[
\bb P^\beta_{\pi_\beta} ( \ell_\beta\delta <  \sigma_1 < t )
\sim   \frac {2\,t}{\bar\ell_\beta},
\] 
with $\sigma_1$ as introduced in (i).
\end{proof}

\section{Converge of the transition times to a Poisson point process} 
\label{sec:7}

Let $(T_\beta)_{\beta>0}$ be a sequence meeting the first two conditions in~\eqref{cond_T} and recursively define the sequence of  stopping times $\sigma_0,\sigma_1, \dots $ on $C([0, \infty))$, by setting $\sigma_0 = 0$ and $  \sigma_k  = \inf\{ t \ge T_\beta+\sigma_{k-1} \colon X_{t} = 0 \} ,  \ k\in \bb N $. Further set 
\[
\mc N =  \inf  \{ k \in \bb N \colon X_{T_\beta + \sigma_{k-1}}  X_{T_\beta + \sigma_{k}}  <0 \},
\]
and   $\zeta =  \sigma_{\mc N}$.

\begin{lemma} 
\label{lemma_exp}
When $X$ is sampled according to $\bb P^\beta_{\pi_\beta}$  and $\beta \to \infty$ the sequence of random variables $ ({\bar\ell_\beta}^{-1} \zeta)_{\beta>0}$ converges in law to an exponential random variable with parameter one.
\end{lemma}

\begin{proof} 
By the strong Markov property and the symmetry of $\bb P^\beta_{0}$ with respect to $X\mapsto -X$,  the random variable $\mc N$ has a geometric distribution with parameter $ 1/2$. Again by the strong Markov property the random variables $(\sigma_{k}-\sigma_{k-1})_{k\in \bb N}$ are  independent and independent from $\mc N$. By Theorem~\ref{t:sht2} and the second statement of Lemma~\ref{lemma_staz0}, for each $k\in \bb N$, as $\beta \to\infty$, the sequence of random variables $({\bar\ell_\beta}^{-1} (\sigma_{k}-\sigma_{k-1}))_{\beta>0} $ converges in law to an exponential random variable with parameter two. The statement is achieved recalling that a geometric sum of independent exponential random variables is an exponential random variable. 
\end{proof}

\begin{proof}[Proof of Theorem~\ref{t:thm3}]
We need to show that for each $R>0$ the restriction of $\mu_{\beta, \ell_\beta} \circ \imath^{-1}_{\bar \ell_\beta}$ to $L^p((-R, R))$ converges weakly to the restriction of $\bar \mu$ to $L^p((-R, R))$. Recalling~\eqref{ipell3}, by the first statement of Lemma~\ref{lemma_staz0}, it is enough to show the statement with $\bb P^\beta_{\pi_\beta}$ instead of $\mu_{\beta, \ell_\beta} $.

\smallskip
\noindent
{\it Step 1.} Recalling the sequence of stopping times $(\sigma_k)$ introduced at the beginning of this section, for each $R, \delta>0$,
\[
\lim_{\beta \to \infty}   \bb P_{\pi_\beta}^\beta \Big( \frac{1}{ \bar \ell_\beta} \int_0^{R\bar\ell_\beta}\!  \rmd t\, \Big|X_t - \sum_{k=0} ^\infty \mathrm{sign} (X_{T_\beta + \sigma_k}) \id_{[\sigma_k, \sigma_{k+1})} (t)  \Big|^p >\delta   \Big) = 0, 
\]
where we understand $\mathrm{sign} (0)=0$. 

Arguing as in the proof of Lemma~\ref{lemma_inclusione}, this step follows from Lemmata~\ref{p:1},~\ref{lemma_Pgrande}, Theorem~\ref{t:sht2} and the strong Markov property.

\smallskip
\noindent
{\it Step 2.} Sampling $X$ according to $\bb P_{\pi_\beta}^\beta$, we recursively define the random variables $\mc N_j$, $j=0,1,\dots$, by setting $\mc N_0= 0$ and
\[  
\mc N_j   =    \inf  \{ k > \mc N_j  :   X_{T_\beta + \sigma_{k-1}}  X_{T_\beta + \sigma_{k}}  <0 \} - \mc N_{j-1}.  
\]
Let also $\zeta_j = \sigma_{\mc N_j}$. Then the random variables $( \zeta_{j} - \zeta_{j-1})_{j}$ are independent. Moreover, for each $j$, $\bar \ell^{-1}_\beta (\zeta_{j}- \zeta_{j-1})$ converges in law to an exponential random variable of parameter one in the limit $\beta \to \infty$. This step follows directly from the strong Markov property and Lemma~\ref{lemma_exp}.

\smallskip
The statement is achieved by combining Steps 1 and 2 together with the symmetry of $\bb P_{\pi_\beta}^\beta$ with respect to $X\mapsto - X$.
\end{proof}

\end{document}